\documentclass[12pt]{article}
\usepackage{exscale,relsize}
\usepackage{amsmath}
\usepackage{amsfonts}
\usepackage[hidelinks]{hyperref}
\usepackage{amssymb}
\usepackage{calc}
\usepackage{theorem}
\usepackage{pifont}      
\usepackage{array}
\usepackage{color}
\usepackage{enumerate}
\usepackage{bbm}
\usepackage{graphicx}
\usepackage{float}
\usepackage{subfigure}

\newcommand{\scal}[2]{\langle{{#1},{#2}}\rangle}

\newcommand{\RR}{\ensuremath{\mathbb R}}

\newcommand{\RPP}{\ensuremath{\left(0,+\infty\right)}}
\newcommand{\RX}{\ensuremath{\left(-\infty,+\infty\right]}}

\newcommand{\NN}{\ensuremath{\mathbb N}}

\newcommand{\menge}[2]{\big\{{#1} \mid {#2}\big\}}

\newcommand{\dom}{\ensuremath{\operatorname{dom}}}

\newcommand{\gra}{\ensuremath{\operatorname{gra}}}

\newcommand{\prox}{\ensuremath{\operatorname{P}_{\mu}}}
\newcommand{\refl}{\ensuremath{\operatorname{R}_{\mu}}}

\newcommand{\inte}{\ensuremath{\operatorname{int}}}

\newcommand{\rank}{\ensuremath{\operatorname{rank}}}

\newcommand{\Fix}{\ensuremath{\operatorname{Fix}}}

\newcommand{\Id}{\ensuremath{\operatorname{Id}}}

\newcommand{\BB}{\ensuremath{\mathbb{B}}}

\renewcommand{\phi}{\ensuremath{\varphi}}

\newtheorem{theorem}{Theorem}[section]
\newtheorem{lemma}[theorem]{Lemma}
\newtheorem{fact}[theorem]{Fact}
\newtheorem{corollary}[theorem]{Corollary}
\newtheorem{proposition}[theorem]{Proposition}
\newtheorem{definition}[theorem]{Definition}

\theoremstyle{plain}{\theorembodyfont{\rmfamily}
}
\theoremstyle{plain}{\theorembodyfont{\rmfamily}
}
\theoremstyle{plain}{\theorembodyfont{\rmfamily}
}
\theoremstyle{plain}{\theorembodyfont{\rmfamily}
\newtheorem{example}[theorem]{Example}}
\theoremstyle{plain}{\theorembodyfont{\rmfamily}
\newtheorem{remark}[theorem]{Remark}}
\theoremstyle{plain}{\theorembodyfont{\rmfamily}
}

\newcommand{\pluss}{{\hskip1pt \raise1pt\vbox{\hrule width6pt \vskip1pt
\hrule width6pt}\kern-4pt{\lower1pt\hbox{\vrule height6pt \kern1pt\vrule
height6pt}}\hskip5pt}}

\newcommand{\argmin}{\mathop{\rm argmin}\limits}

\usepackage{tikz}

\begin{document}

\title{{\fontfamily{ptm}\selectfont Modulus of conically averaged mappings and its applications to angles between two
subspaces}}

\author{
Honglin Luo\thanks{School of Mathematical Sciences, Chongqing Normal University, Chongqing, PRC. Email: \texttt{071025013@fudan.edu.cn}},
~Shuang Song\thanks{School of Mathematics and Statistics, The University of Melbourne,
Parkville, VIC 3010, Australia. E-mail: \texttt{ssong19249@student.unimelb.edu.au}
}, and Xianfu Wang\thanks{Department of Mathematics, I.K. Barber Faculty of Science,
The University of British Columbia, Kelowna, BC Canada V1V 1V7. E-mail:  \texttt{shawn.wang@ubc.ca}} }

\date{May 15, 2026 (Revision)}

\maketitle

\vskip 8mm

\begin{abstract} \noindent Conically averaged mappings, a generalization of averaged mappings, are important in a wide range of Optimization Algorithms.
In this paper, we propose the modulus of conical averagedness to classify conical averaged mappings.
Introducing the monotone and comonotone values of generalized monotone mappings,
we investigate their connections to the modulus of conical averagedness.
In the linear setting, we completely characterize conically
averaged matrices, and derive explicit and pleasing formulae for computing their modulus
of averagedness.
As applications, we compute the Dixmier and Friedrichs angles between
two subspaces. Nonlinear results are established as extensions of the linear case.
Conical averagedness of proximal and reflection mappings
of hypoconvex functions are also studied.
\end{abstract}

{\small



\section{Introduction}
Throughout, we assume that
$$
X \text { is a real Hilbert space with inner product }\langle\cdot, \cdot\rangle: X \times X \rightarrow \mathbb{R} \text {, }
$$ and induced norm $\|\cdot\|$. Let $\Id$ denote the identity operator on $X$.
Recall the following well-known notion of nonexpansiveness and its variants \cite{BC, cegielski}, which play a central role in applied mathematics.
\begin{definition}
Let $T: X \rightarrow X$ and let $\mu>0$. Then $T$ is
\begin{enumerate}
\item nonexpansive\footnote{For convenience, we shall assume that
$T$ has a full domain throughout the paper while
one can generalize it to be a proper subset of $X$.} if $$
(\forall x \in X)(\forall y \in X) \quad\|T x-T y\| \leq\|x-y\|;
$$
\item $\mu$-cocoercive if
$$
(\forall x \in X)(\forall y \in X) \quad \langle x-y, T x-T y\rangle \geq \mu\|T x-T y\|^2;
$$
\item $k$-averaged
if $T=(1-k) \Id+k N$ for some nonexpansive operator $N$ and some $k \in [0,1]$.
\end{enumerate}
The set of fixed points of $T$ is denoted by $\Fix T :=\menge{x\in X}{Tx=x}.$
\end{definition}

Averaged mappings are broad and important in optimization;
see, e.g., \cite{BBR, BBM, BC, BM, BMW, nollphan, cegielski, combettes,  CY, OY, XH}. The following definition naturally extends the concept of an averaged mapping.

\begin{definition}\emph{\cite[Definition 3.1]{BMW}(see also \cite[Definition 2.1]{BDP})}
Let $T: X \rightarrow X$. Then $T$ is conically $k$-averaged, if $T=(1-k) \Id+k N$ for some nonexpansive operator $N:X\rightarrow X$, and $k \in[0,+\infty)$.
\end{definition}

Due to applications in nonconvex optimization, conically averaged mappings are very active topics in recent years;
see, e.g., \cite{BDP,BMW, evens, PG, PG1, XH}.
In particular, recent exemplary work by Bartz, Dao and Phan \cite{BDP} fully demonstrates
how one can deploy conical averaged mappings
in various optimization algorithms.
Note that conical averagedness is also referred to as conical nonexpansiveness. Similarly to averagedness, when $k \in (0,+\infty)$, various characterizations of conically $k$-averagedness are available, including (see \cite[Proposition 2.2]{BDP})
\begin{equation}\label{e:average1}
(\forall x \in X)(\forall y \in X) \quad\|T x-T y\|^2 \leq\|x-y\|^2-\frac{1-k}{k}\|(\operatorname{Id}-T) x-(\operatorname{Id}-T) y\|^2,
\end{equation}
and
\begin{equation}\label{e:average2}
(\forall x \in X)(\forall y \in X) \quad\|T x-T y\|^2+(1-2 k)\|x-y\|^2 \leq 2(1-k)\langle x-y, T x-T y\rangle.
\end{equation}

From \eqref{e:average1} or \eqref{e:average2} and the fact that $\mathrm{Id}$ is the only conically
$0$-averaged operator, we can deduce that if an
operator is conically $k_0$-averaged, then it is conically $k$-averaged for every $k \geq k_0$. Motivated by this, Bauschke, Bendit and Moursi \cite{BBM} proposed the \emph{modulus of averagedness}, defined as the minimum averaged constant of an averaged mapping. Now we naturally extend this concept to a conical version.

\begin{definition}[modulus of conically averaged mapping]
Let $T: X \rightarrow X$. Then the modulus of conical averagedness of $T$ is defined by
$$
k(T):=\inf \{k \in[0,+\infty) \mid T \text { is } \text{conically $k$-averaged }\}.
$$
\end{definition}

Note that if $T$ is not conically $k$-averaged for any $k \in[0,+\infty)$, then $k(T)=+\infty$ since $\inf \varnothing=+\infty$. We say that $T: X \rightarrow X$ is conically averaged, if it is conically $k$-averaged for some $k \in[0,+\infty)$, equivalently, $k(T)<+\infty$. Finding $k(T)$ is an optimization problem by nature.
\begin{proposition} Let $T:X\rightarrow X$ be conically averaged with $k(T)>0$. Then
$$\left. k(T)=\sup\left\{\frac{1}{2}\frac{\|Tx-Ty-(x-y)\|^2}{\|x-y\|^2-\scal{x-y}{Tx-Ty}}\right\vert\ \|x-y\|^2\neq \scal{x-y}{Tx-Ty}, x,y\in X\right\}.$$
\end{proposition}
\begin{proof} Rewrite \eqref{e:average2} as
$\|T x-T y-(x-y)\|^2 \leq 2 k\left(\|x-y\|^2-\langle x-y, T x-T y\rangle\right).$
\end{proof}

\emph{The goal of this paper is to study the modulus of conically averaged mappings.
As a key contribution, the conical averagedness of linear mappings is systematically analyzed.
We give explicit formulae to compute the modulus of a conically averaged matrix.
An amazing application is the algorithmic computation of the angle between two subspaces.
We also study the modulus of nonlinear conically averaged mappings.
They allow us to derive different formulae for angles between
subspaces. Many examples are provided to illustrate our
results. These results are  new even for averaged mappings.}

The rest of the paper is organized as follows. Section~\ref{s:auxiliary r} provides auxiliary results on
conically averaged mappings. Section~\ref{s:monotonev} introduces monotone and comonotone values for generalized monotone mappings, and highlights their connections to the modulus of conical averagedness.
In Section~\ref{s:linearw} we investigate conical averagedness in the linear setting and present
remarkable characterizations of conically averaged matrices, along with formulae for computing the modulus
of conical averagedness. It turns out that a conically averaged matrix lies in the interior of the set of
all conically averaged matrices if and only if the maximal eigenvalue of its symmetric part is less than one.
 In Section~\ref{s:angleb}, we explore the Dixmier and Friedrichs angles
between two subspaces using the modulus of averagedness, and propose new computational methods.
In Section~\ref{s:nonlinear}, we establish further nonlinear results that serve as generalizations or alternatives to our earlier results on matrices. Finally, in Section~\ref{s:hypoconvex}
we apply results in earlier sections to hypoconvex functions.

\section{Auxiliary results}\label{s:auxiliary r}

This section collects some preparatory results on
the modulus of conical averagedness that will be used in later proofs.

We start with a simple example showing that $T\mapsto k(T)$ is an extended nonnegative-valued function,
i.e., taking values in $[0,+\infty]$.
\begin{example}\label{e:Id}
Let $\alpha \in \RR$. Then
\begin{equation}\label{e:id:case}
k(\alpha \mathrm{Id})=\begin{cases}
\frac{1-\alpha}{2} & \text{if $\alpha \leq 1$,} \\
+\infty & \text{if $\alpha>1$.}
\end{cases}
\end{equation}
\end{example}
\begin{proof} Write $\alpha \Id=(1-\lambda)\Id+\lambda N$ with $\lambda\in [0,+\infty)$ and $N:X\rightarrow X$ being nonexpansive. If $\lambda=0$, we have $\alpha=1$ and $k(\Id)=0$, so \eqref{e:id:case} holds.
We only need to consider $\lambda>0$ case.
Now
\begin{equation}
N=\frac{\lambda+\alpha-1}{\lambda}\Id,
\end{equation}
and $N$ being nonexpansive requires
\begin{equation}\label{e:nonexp}
\left|\frac{\lambda+\alpha-1}{\lambda}\right|\leq 1.
\end{equation}
This gives $\alpha\leq 1 $ and $\lambda\geq (1-\alpha)/2$. Taking infimum over $\lambda$ yields $k(\alpha\Id)=(1-\alpha)/2$ when $\alpha\leq 1$.
If $\alpha>1$, then \eqref{e:nonexp} never holds, so $k(\alpha \Id)=\inf(\varnothing)=+\infty$.
\end{proof}

\begin{remark} Example~\ref{e:Id} implies that
\begin{enumerate}
\item
$T$ being conically averaged does not imply that $-T$ is conically averaged. Take $T=-2\Id$.
\item $T\mapsto k(T)$ is not positively homogeneous, e.g., $k(2\Id)=+\infty$, but $2k(\Id)=0$.
\end{enumerate}
\end{remark}

Rewriting \eqref{e:average2} we find that the following characterization of conical averaged mappings is more convenient to use in this paper.
\begin{lemma}\label{l:char:con}
Let $T:X\rightarrow X$ and let $k \in(0,+\infty)$. Then $T$ is conically $k$-averaged if and only if
\begin{equation}\label{e:average3}
(\forall x \in X)(\forall y \in X) \quad\|T x-T y-(x-y)\|^2 \leq 2 k
\left(\|x-y\|^2-\langle x-y, T x-T y\rangle\right).
\end{equation}
\end{lemma}

For any operator $T: X \rightarrow X$ and any $v \in X$, the operator $T+v$ is defined by
$$
(\forall x \in X) \quad(T+v) x:=T x+v.
$$
The following result extends \cite[Propositions 2.1, 2.2 and 2.3]{SW} from averaged mappings to conically averaged
mappings. We omit the proofs, because they are similar to those in \cite{SW}.
\begin{lemma}\emph{}\label{zero case}
Let $T: X \rightarrow X$ be conically averaged. Then the following hold:
\begin{enumerate}
\item For any $v \in X$, $k(T+v)=k(T)$ and $k(T(\cdot+v))=k(T)$.
\item If $k(T)>0$, then $T$ is conically $k(T)$-averaged.
\item $k(T)=0$ if and only if $T=\operatorname{Id}+v$ for some $v \in X$.
\end{enumerate}
\end{lemma}

\begin{corollary}\label{c:char:con}
Let $T:X\to X$ and let $K\in[0,+\infty)$. Suppose that
$$
(\forall x \in X)(\forall y \in X)\quad
\|Tx-Ty-(x-y)\|^2
\le
2K\big(\|x-y\|^2-\langle x-y,Tx-Ty\rangle\big).
$$
Then $k(T)\le K$.
\end{corollary}

\begin{proof}
If $K>0$, then $T$ is conically $K$-averaged by Lemma~\ref{l:char:con},
and hence $k(T)\le K$. If $K=0$, then
$(\forall x,y\in X)\ Tx-Ty=x-y.
$
Thus $T=\Id+v$ for some $v\in X$, and Lemma~\ref{zero case} yields $k(T)=0$.
Hence $k(T)\le K$ also in this case.
\end{proof}

Nonexpansiveness of a conically averaged mapping can be quantified by its modulus of averagedness.
\begin{fact}\label{f:one:lip} Let $T: X \rightarrow X$ be conically averaged.
Then $T$ is nonexpansive if and only if $k(T)\leq 1$.
\end{fact}

The following fact concerns the convex combination of conically averaged operators, which generalizes the result obtained by Ogura and Yamada \cite{OY}.
\begin{fact}\emph{\cite[Proposition 2.4]{BDP}}\label{BDPconvexfact}
Let $I$ be a finite index set. For each $i \in I$ let $T_i: X \rightarrow X$ be conically $\theta_i$-averaged. Let $\left\{\omega_i\right\}_{i \in I} \subseteq \mathbb{R}_{++}$ with $\sum_{i \in I} \omega_i=1$. Then $\sum_{i \in I} \omega_i T_i$ is conically $\theta$-averaged where $\theta:=\sum_{i \in I} \omega_i \theta_i$.
\end{fact}

\begin{proposition}\label{newOYprop}
Let $T_i: X \rightarrow X$ be conically averaged for $i=1, 2$ and $\lambda \in [0, 1]$. Then $$k\left((1-\lambda) T_1+\lambda T_2\right) \leq(1-\lambda) k\left(T_1\right)+\lambda k\left(T_2\right).$$
\end{proposition}

\begin{proof}
Apply Fact \ref{BDPconvexfact}.
\end{proof}

\begin{corollary}\label{scalarmul}
Let $T: X \rightarrow X$ be conically averaged and let $\alpha \in[0,1]$. Then $k(\alpha T) \leq \alpha k(T)+(1-\alpha) / 2 .$
\end{corollary}

\begin{proof}
Let $T_1=0$ (the zero mapping) in Proposition \ref{newOYprop} and apply the fact that $k(0)=1/2$ (see \cite[Example 2.5]{SW}).
\end{proof}

\begin{remark}
\begin{enumerate}
\item Corollary~\ref{scalarmul} fails for $\alpha>1$. Take $\alpha=2, T=\Id$.
Then $k(2\Id)=+\infty$, but $2 k(\Id)+(-1/2)=2\cdot 0-1/2=-1/2$.
\item
Corollary~\ref{scalarmul} is sharp. Take $\alpha=0$. Then $k(0T)=k(0)=1/2= 0k(T)+1/2$.
\end{enumerate}
\end{remark}

Next is a basic operation that preserves conical averagedness.
\begin{lemma}\label{l:minus}
Let $T: X \rightarrow X$ be conically averaged. Then
$T-\varepsilon\Id$ is conically averaged for every $\varepsilon>0$ and
$k(T-\varepsilon \Id) \leq k(T)+\varepsilon$.
\end{lemma}
\begin{proof} By the assumption, $T=(1-\alpha)\Id+\alpha N$ for some $\alpha\in [0, +\infty)$ and a nonexpansive mapping
$N:X\rightarrow X$.
For every $\varepsilon>0$, we have
$$
\begin{aligned}
T-\varepsilon \Id & =(1-\alpha) \Id+\alpha N-\varepsilon \Id \\
& =(1-\alpha-\varepsilon) \Id+\alpha N \\
& =[1-(\alpha+\varepsilon)] \Id+\alpha N \\
& =[1-(\alpha +\varepsilon)] \Id+(\alpha+\varepsilon) \cdot\left(\frac{\alpha}{\alpha+\varepsilon} N\right),
\end{aligned}
$$
which implies that $T-\varepsilon \mathrm{Id}$ is conically $(\alpha+\varepsilon)$-averaged, i.e.,
$k(T-\varepsilon\Id)\leq \alpha+\varepsilon$. Thus, $k(T-\varepsilon \Id) \leq k(T)+\varepsilon$ by taking infimum over $\alpha\geq k(T)$.
\end{proof}


Finally, we recall that if $f\le g$, then $\varliminf f\le \varliminf g$ and $\varlimsup f\le \varlimsup g$. Moreover, we have the following classical analysis result on the limit inferior and limit superior of functions.
\begin{fact}\label{f:l:supinf}
Let $f, g:\RR\rightarrow\RR$ and $\alpha_{0}\in\RR$. Suppose that $\lim _{\alpha \rightarrow \alpha_0} g(\alpha)=\beta$ exists.
\begin{enumerate}
\item If $\varliminf_{\alpha \rightarrow \alpha_0} f(\alpha)\in \RR$, then
$\varliminf_{\alpha\rightarrow\alpha_{0}} (f(\alpha)g(\alpha)) \leq \varliminf_{\alpha\rightarrow\alpha_{0}} f(\alpha)\cdot \beta.$
\item If $\varlimsup_{\alpha \rightarrow \alpha_0} f(\alpha)\in \RR$, then
$\varlimsup_{\alpha\rightarrow\alpha_{0}} (f(\alpha)g(\alpha)) \geq \varlimsup_{\alpha\rightarrow\alpha_{0}} f(\alpha)\cdot \beta.$
\end{enumerate}
\end{fact}

\section{Conical averaged mapping and generalized monotonicity}\label{s:monotonev}

Let $A: X \rightrightarrows X$ be a set-valued mapping.
Recall that the graph of $A$ is $\gra A :=\{(x, u) \in X \times X \mid u \in A x\}$ and
the inverse of $A$, denoted by $A^{-1}$, is the operator with
graph $\operatorname{gra} A^{-1} :=\{(u, x) \in X \times X \mid u \in A x\}$. The domain of $A$ is $\operatorname{dom} A :=\{x \in X \mid A x \neq \varnothing\}$. The resolvent and reflected resolvent of $A$ are $J_A :=(\mathrm{Id}+A)^{-1}$ and
$R_{A}:=2J_{A}-\Id$, respectively. For $\mu \in \RR$, the Yosida $\mu$-regularization of $A$ is the operator $Y_\mu(A):=\left(\mu \mathrm{Id}+A^{-1}\right)^{-1}$. $A$ is monotone, if
$$
(\forall(x, u)\in \gra A)(\forall (y, v) \in \operatorname{gra} A) \quad\langle x-y, u-v\rangle \geq 0.
$$
$A$ is maximally monotone, if it is monotone and there is no monotone operator $B: X \rightrightarrows X$ such that $\gra B$ properly contains $\gra A$. For more information on monotone mappings, see \cite{BC,cegielski,RW}.
Conical averaged mappings are closely related to generalized monotone mappings defined below.


\subsection{Generalized monotonicity}

Recall the following definitions of generalized monotone mappings; see, e.g., \cite{BDP},
\cite[Definition 2.4]{BMW},  \cite[Definition 1.2]{iusem23}, \cite{RY}, and \cite{XH}.

\begin{definition}
Let $A: X \rightrightarrows X$ and let $\rho \in \mathbb{R}$. We define the following notions:
\begin{enumerate}
\item
$A$ is $\rho$-monotone if $(\forall(x, u) \in \operatorname{gra} A)(\forall(y, v) \in \operatorname{gra} A)$ we have
\begin{equation}\label{e:rho:mono}
\langle x-y, u-v\rangle \geq \rho\|x-y\|^2 .
\end{equation}
\item $A$ is maximally $\rho$-monotone if $A$ is $\rho$-monotone and there is no $\rho$-monotone operator $B: X \rightrightarrows X$ such that $\gra B$ properly contains $\gra A$, i.e., for every $(x, u) \in X \times X$,
$$
(x, u) \in \operatorname{gra} A \Leftrightarrow [(\forall(y, v) \in \operatorname{gra} A)\ \langle x-y, u-v\rangle \geq \rho\|x-y\|^2].
$$
\item $A$ is $\rho$-comonotone if $(\forall(x, u) \in \operatorname{gra} A)(\forall(y, v) \in \operatorname{gra}A)$ we have
$$
\langle x-y, u-v\rangle \geq \rho\|u-v\|^2 .
$$
\item $A$ is maximally $\rho$-comonotone if $A$ is $\rho$-comonotone and there is no $\rho$-comonotone operator $B: X \rightrightarrows X$ such that $\gra B$ properly contains $\gra A$, i.e., for every $(x, u) \in X \times X$,
$$
(x, u) \in \operatorname{gra} A \Leftrightarrow [(\forall(y, v) \in \operatorname{gra} A)\ \langle x-y, u-v\rangle \geq \rho\|u-v\|^2].
$$
\end{enumerate}
\end{definition}

Note that when $\rho>0$, a $\rho$-comonotone operator $A$ is at most single-valued, and $\rho$-comonotonicity reduces to $\rho$-cocoercivity on a subset of $X$. If the $\rho$-comonotonicity is maximal, this reduction holds on the entire space $X$, allowing us to write $A: X \rightarrow X$. This follows from the following result, which generalizes
\cite[Theorem 15, page 221]{RY} from $\RR^n$ to Hilbert spaces via Minty's Theorem \cite[Theorem 21.1]{BC}.

\begin{lemma}\label{maxcomonofact}
Let $A: X \rightrightarrows X$ be $\rho$-comonotone with $\rho>0$. Then $A$ is maximally $\rho$-comonotone if and only if $\operatorname{dom}A=X$.
\end{lemma}

Observe that if $A$ is $\rho_0$-(co)monotone with $\rho_0 \in \RR$, then it is $\rho$-(co)monotone for any $\rho \leq \rho_0$. This motivates us to define the following values for a set-valued mapping.

\begin{definition}\emph{(\textbf{monotone value and comonotone value})}\label{d:mono:val}
Let $A: X \rightrightarrows X$. Then the monotone value of $A$ is defined by
$$
m(A):=\sup \{\rho \in \mathbb{R} \mid A \text { is } \rho \text {-monotone }\}.
$$
The comonotone value of $A$ is defined by
$$
c(A):=\sup \{\rho \in \mathbb{R} \mid A \text { is } \rho \text {-comonotone }\} .
$$
\end{definition}

\begin{remark}
When $A$ is maximally monotone, $m(A)$ coincides with the monotone value defined in \cite[Definition 6.1]{SW}, and $c(A)$ coincides with the cocoercive value \cite[Definition 6.2]{SW} in view of Lemma~\ref{maxcomonofact}.
\end{remark}

Basic properties of the monotone and comonotone values come as follows.

\begin{proposition}\label{bigpropforcm}
Let $A, B: X \rightrightarrows X$ with $m(A), m(B), c(A), c(B)>-\infty$. Let $\alpha > 0$ and $\mu \in \RR$. Then
the following hold:
\begin{enumerate}
\item\label{i:dualthing}
\emph{(\textbf{duality})} $c(A)=m\left(A^{-1}\right)$ and $c\left(A^{-1}\right)=m(A)$.
\item \label{i:mono2} $A$ is monotone if and only if $m(A) \geq 0$, if and only if $c(A) \geq 0$.
\item \label{i:mono3} $m(\alpha A)=\alpha m(A)$ and $c(\alpha A)=\alpha^{-1} c(A)$.
\item \label{i:mono4} $m(A+B) \geq m(A)+m(B)$ and $$c\left(\left(A^{-1}+B^{-1}\right)^{-1}\right) \geq c(A)+c(B).$$
\item \label{i:mono5}
$m(A+\mu \Id)=m(A)+\mu$ and $c\left(Y_\mu(A)\right)=c(A)+\mu.$
\end{enumerate}
\end{proposition}

\begin{proof}
\ref{i:dualthing}, \ref{i:mono2}, \ref{i:mono3} and \ref{i:mono4} can be directly verified. \ref{i:mono5}:
Since $Y_\mu(A)=\left(\mu \mathrm{Id}+A^{-1}\right)^{-1}$, we have
$$
\begin{aligned}
c\left(Y_\mu(A)\right)
& =c\left((\mu \mathrm{Id}+A^{-1}\right)^{-1})
=m\left(\mu \mathrm{Id}+A^{-1}\right) \\
& =\mu+m\left(A^{-1}\right)
 =\mu+c(A).
\end{aligned}
$$
\end{proof}

\begin{example}
\begin{enumerate}
\item
Let $\alpha \in \mathbb{R}$. Then $m\left(\alpha \Id\right)=\alpha$ and$$
c(\alpha \Id)=\begin{cases}
\alpha^{-1} & \text{if $\alpha \neq 0$,} \\
+\infty & \text{if $\alpha=0$.}
\end{cases}
$$
\item
Let $f(x)=e^x$ on $\RR$. Then $m(f)=c(f)=0$.
\item
Let $A: X \rightrightarrows X$. Then $c(A)=+\infty$ if and only if there exists $v \in X$ such that $Ax=\{v\}$ for any $x \in \operatorname{dom}A$.
\end{enumerate}
\end{example}

\subsection{Connection with conical averagedness}

In this subsection, we provide useful formulae for computing the modulus via
the cocoercivity and comonotone values.
An earlier version of the following fact on averaged mappings goes back to \cite[Proposition 3.4]{XH}.

\begin{fact}\label{kcfact}\emph{\cite[Corollary 3.5]{BMW}}
Let $T: X \rightarrow X$, and let $\alpha \in(0,+\infty)$. Then
$$
T \text { is } \text {conically $\alpha$-averaged } \Leftrightarrow \mathrm{Id}-T \text { is } \frac{1}{2 \alpha} \text {-cocoercive. }
$$
\end{fact}

\begin{proposition}[modulus of averagedness via comonotone value]\label{kcprop!}
Let $T: X \rightarrow X$ be conically averaged. Then
\begin{equation}
k(T)=\frac{1}{2} \frac{1}{c\left(\mathrm{Id}-T\right)}.
\end{equation}
\end{proposition}
\begin{proof}
Combine Fact \ref{kcfact} and Lemma~\ref{zero case}. Indeed,
if $k(T)=0$, Lemma~\ref{zero case} gives $T=\Id+v$ for some $v\in X$,
implying $\Id-T\equiv v$ so that $c(\Id-T)=+\infty$. Then $k(T)=0=1/(2\infty)=1/(2c(\Id-T))$.
If $k(T)>0$, Fact~\ref{kcfact} gives that $T$ is $\alpha$-averaged if and only if
$\Id-T$ is $1/(2\alpha)$ cocoercive. Taking infimum all $\alpha\geq k(T)$ yields
$1/(2k(T))=c(\Id-T)$, as required.
\end{proof}

To study the modulus of resolvents of comonotone mappings, the following facts help.

\begin{fact}\emph{(\textbf{generalized Minty's Theorem})}\emph{\cite[Theorem 2.17]{BMW}}
Let $A: X \rightrightarrows X$ be $\rho$-comonotone for some $\rho>-1$. Then $A$ is maximally $\rho$-comonotone $\Leftrightarrow \operatorname{ran}(\operatorname{Id}+A)=X$.
\end{fact}

\begin{fact}\emph{\cite[Corollary 2.14]{BMW}(see also \cite[Proposition 3.4]{DP})}
Let $A: X \rightrightarrows X$ be maximally $\rho$-comonotone for some $\rho>-1$. Then $J_A$ and $J_{A^{-1}}$ are single-valued and with full domain.
\end{fact}

Therefore, for a maximally $\rho$-comonotone operator $A$ with $\rho>-1$, we may write $J_A: X \rightarrow X$. In what follows, we will typically assume $A$ to be a maximally $\rho$-comonotone operator with $\rho>-1$ instead of a nonmaximal one, as a maximal extension always exists by Zorn’s lemma.

The intimate relationship between comonotone mappings and conically averaged mappings is given in the following fundamental fact.

\begin{fact}\label{JAcAcomonofact}\emph{\cite[Corollary 3.8]{BMW}}
Let $T: X \rightarrow X$ and $\alpha \in ( 0,+\infty)$. Then $T$ is conically $\alpha$-averaged if and only if it is the resolvent of a maximally $\rho$-comonotone operator $A: X \rightrightarrows X$, where $\rho=\frac{1}{2 \alpha}-1>-1$, i.e., $\alpha=\frac{1}{2(\rho+1)}$.
\end{fact}

\begin{proposition}[modulus of averagedness of resolvent via comonotone value]\label{JACACOMONOPROP}
Let $A: X \rightrightarrows X$ be maximally $\rho$-comonotone with $\rho >-1$, i.e., $c(A)>-1$. Then $$
k\left(J_A\right)=\frac{1}{2} \frac{1}{1+c(A)}.
$$ Consequently, if $-1<c(A)<-1/2$, then $J_{A}$ is not nonexpansive.
\end{proposition}

\begin{proof}
Apply Fact~\ref{JAcAcomonofact}.
If $-1<c(A)<-1/2$, then $2(1+c(A))<1$ so that $k(J_{A})>1$, thus
$J_{A}$ is not nonexpansive.
\end{proof}

We illustrate Proposition~\ref{JACACOMONOPROP} with an example.
\begin{example} Let $\alpha<-1$ and $A:=\alpha\Id$. Then $A^{-1}=\alpha^{-1}\Id$ is $(\alpha^{-1})$-monotone
with $\alpha^{-1}>-1$,
i.e., $A$ is comonotone with $c(A)=\alpha^{-1}>-1$. We have
\begin{equation}
J_{A}=\frac{1}{1+\alpha}\Id, \quad J_{A^{-1}}=\frac{\alpha}{1+\alpha}\Id.
\end{equation}
While
\begin{equation} k(J_{A^{-1}})=+\infty
\end{equation}
because of $\alpha/(1+\alpha)>1$,
\begin{equation}
k(J_{A})=\frac{1}{2}\frac{1}{1+\alpha^{-1}}=\frac{1}{2}\frac{\alpha}{1+\alpha}.
\end{equation}
Thus $J_{A}$ is conically averaged, but $J_{A^{-1}}$ is not.
Moreover, $J_{A}$ is nonexpansive when $\alpha\leq -2$, and expansive when $-2<\alpha<-1$. However,
$J_{A^{-1}}$ is always expansive when $\alpha<-1$.
\end{example}

\begin{corollary}
Let $A, A^{-1}: X \rightrightarrows X$ be maximally $\rho$-comonotone with $\rho>-1$. Then
$k\left(J_A\right)=k\left(J_{A^{-1}}\right)$ if and only if $c(A)=m(A)$.
\end{corollary}
\begin{proof}
Apply Proposition~\ref{JACACOMONOPROP} to both $A$ and $A^{-1}$, and use $c\left(A^{-1}\right)=m(A)$ in
Proposition~\ref{bigpropforcm}\ref{i:dualthing}.
\end{proof}

\begin{corollary} Let $A: X \rightrightarrows X$ be $\rho$-monotone with $\rho>-2$.
Then $k\left(J_{J_A}\right)=\frac{1}{2} \frac{1}{2+m(A)}$.
\end{corollary}

\begin{proof}
By Proposition \ref{bigpropforcm}, we have $c\left(J_A\right)=m(\Id+A)=1+m(A)$.
This implies that $J_{A}$ is $\rho$-comonotone with $\rho>-1$ by our assumption.
Then apply Proposition~\ref{JACACOMONOPROP}
to $J_{A}$.
\end{proof}

\section{Linear world}\label{s:linearw}

In this section, we discuss conically averaged mappings in the linear setting.
For simplicity, we will now work in Euclidean space $\mathbb{R}^n$ for some $n \in \NN$, although most results remain valid in a general Hilbert space setting.
We give conditions under which
a matrix is a conically averaged mapping, and provide surprising formulae for the modulus of averagedness of
a conically averaged matrix
in terms of its eigenvalues. Moreover, we show that the function of modulus of
conical averagedness is proper, lower
semicontinuous, and convex; and that the interior of the set of conically averaged matrices consists of
the matrices whose symmetric parts have their largest eigenvalues less than one.

Let $M_n(\mathbb{R})$ denote the set of all $n\times n$ square matrices, $\mathbb{S}^n:=\left\{A \in M_n(\mathbb{R}) \mid A^{\top}=A\right\}$ the set of all symmetric matrices, and $\sigma(A)$ the set of all eigenvalues of $A$ (the spectrum of $A$).
We will use $A_S:=(A+A^{\top})/2$ for the symmetric part of $A$, and $\lambda_{\min}(A)$ (resp., $\lambda_{\max}(A)$) for the smallest (resp., largest) eigenvalue of $A$ provided that all eigenvalues of $A$ are real. The matrix $2$-norm of $A\in M_{n}(\RR)$ is
$\|A\|_2:=\sqrt{\lambda_{\max}(A^{\top}A)}$. See \cite[section 5.2]{meyer} for more details.

\subsection{Symmetric or skew matrix}

Symmetric matrices form an important subclass of square matrices and their modulus is of independent interest.

\begin{proposition}\label{kSn}
Let $A \in \mathbb{S}^n$. Then $A$ is conically averaged if and only if $\sigma(A) \subset (-\infty,1]$, in which case the formula holds:$$
k(A)=\frac{1-\lambda_{\min }(A)}{2}.
$$
Consequently, if $\lambda_{\max}(A)>1$ we have $k(A)=+\infty$.
\end{proposition}

\begin{proof}
By definition $A$ is conically averaged if and only if $A=(1-\alpha) \Id +\alpha N$ for some $\alpha \geq 0$ and nonexpansive operator $N$. Since $A \in \mathbb{S}^n$, this is equivalent to
$$
\begin{aligned}
& \|A-(1-\alpha) \Id\|_{2} \leq \alpha \\
\Leftrightarrow &\ \sqrt{\lambda_{\max }\left((A-(1-\alpha) \Id)^2\right)} \leq \alpha \\
\Leftrightarrow &\ \lambda_{\max }\left(A^2-2(1-\alpha) A+(1-\alpha)^2 \Id\right) \leq \alpha^2 \\
\Leftrightarrow &\ \lambda_{\max }\left(A^2-2(1-\alpha) A+(1-2 \alpha) \Id\right) \leq 0 \\
\Leftrightarrow &\ \max _{\lambda \in \sigma(A)}\left(\lambda^2-2(1-\alpha) \lambda+1-2 \alpha\right) \leq 0.
\end{aligned}
$$
Note that in the last equivalence we use the fact that the matrix
$A^2-2(1-\alpha)A+(1-2\alpha)\Id$ has eigenvalues of form
$\lambda^2-2(1-\alpha)\lambda+(1-2\alpha)$ with $\lambda\in \sigma(A)$ because of $A\in \mathbb{S}^n$.
Now
$$
\lambda^2-2(1-\alpha) \lambda+1-2 \alpha=(\lambda-1)(\lambda-(1-2 \alpha)) .
$$
Since $\alpha \geq 0$, one has $1-2 \alpha \leq 1$, and the above quadratic is nonpositive exactly on the interval $[1-2 \alpha, 1]$. Therefore,
$$
\max _{\lambda \in \sigma(A)}\left(\lambda^2-2(1-\alpha) \lambda+1-2 \alpha\right) \leq 0
$$
is equivalent to
$$
1-2 \alpha \leq \lambda \leq 1 \quad \text { for every } \lambda \in \sigma(A).
$$
This yields $\sigma(A) \subset(-\infty, 1]$ and $\alpha \geq\left(1-\lambda_{\min }(A)\right) / 2$, as required.
\end{proof}

\begin{remark} Proposition~\ref{kSn} significantly extends Example~\ref{e:Id}.
\end{remark}

For skew-symmetric matrices, we have the following result.

\begin{proposition}\label{p:skew}
Let $A \in M_n(\mathbb{R})$ be skew-symmetric, i.e., $A^{\top}=-A$. Then $A$ is always conically averaged, and
$$
k(A)=\frac{1-\lambda_{\min }(A^2)}{2} .
$$
\end{proposition}

\begin{proof}
For $\alpha \geq 0$ we have
$$
\begin{aligned}
& \|A-(1-\alpha) \mathrm{Id}\|_{2} \leq \alpha \\
\Leftrightarrow &\ \sqrt{\lambda_{\max }\left((A^{\top}-(1-\alpha) \mathrm{Id})(A-(1-\alpha) \mathrm{Id})\right)} \leq \alpha \\
\Leftrightarrow &\
 \lambda_{\max }\left(-A^2+(1-\alpha)^2\Id\right) \leq \alpha^2 \\
\Leftrightarrow &\ \lambda_{\max }\left(-A^2\right) \leq 2\alpha-1 \\
\Leftrightarrow &\ -\lambda_{\min }\left(A^2\right) \leq 2\alpha-1 \\
\Leftrightarrow &\ (1-\lambda_{\min }(A^2))/2 \leq \alpha.
\end{aligned}
$$
Note that $A^2$ is symmetric and negative semidefinite since $A$ is skew. Hence $\lambda_{\min }\left(A^2\right) \leq 0$, which implies that there always exists $\alpha \geq 0$ to make the above inequality hold. The chain of equivalences completes the proof.
\end{proof}

\subsection{General matrix $A$: when $\Fix A=\{0\}$}

Observe that $\Fix A=\{0\}$ if and only if $\Id -A$ is invertible, i.e., nonsingular.
We first give the formula to compute the comonotone value of nonsingular matrices. Recall
\begin{fact}\emph{\cite[Proposition 5.1]{BMW}}\label{linearmonofact}
Let $A \in \mathbb{R}^{n \times n}$ and let $\rho \in \mathbb{R}$. Then the following hold:
\begin{enumerate}
\item
$A$ is $\rho$-monotone $\Leftrightarrow \lambda_{\min }\left(A_S\right) \geq \rho$.
\item $A$ is $\rho$-comonotone $\Leftrightarrow \lambda_{\min }\left(A_S-\rho A^{\top} A\right) \geq 0$.
\end{enumerate}
\end{fact}

\begin{proposition}\label{linear mc}
Let $A \in M_n(\mathbb{R})$. Then the following hold:
\begin{enumerate}
\item\label{rho:mono1} $m(A)=\lambda_{\min }(A_S)$.
\item\label{rho:mono2} $c(A)=\lambda_{\min }((A^{-1})_S)$ if $A$ is invertible.
\end{enumerate}
\end{proposition}

\begin{proof}
\ref{rho:mono1}: In view of Definition~\ref{d:mono:val}, we
apply Fact \ref{linearmonofact} and \eqref{e:rho:mono}.

\ref{rho:mono2}:
Combine \ref{rho:mono1} and $c(A)=m\left(A^{-1}\right)$ in Proposition~\ref{bigpropforcm}\ref{i:dualthing}.
\end{proof}

Under the nonsingular condition, i.e., $\Id-A$ being invertible, we can provide a method to determine whether a square matrix is conically
averaged and compute its modulus of conical averagedness. This result also plays a central role
in later Subsection~\ref{s:sub:nonzero}, Section~\ref{s:angleb}, and Subsection~\ref{s:sub:continuity}.

\begin{theorem}\emph{\textbf{(matrix with zero fixed point only)}}\label{nonsingular theorem}
Let $A \in M_n(\mathbb{R})$ and suppose that $\Id-A$ is invertible (i.e., $1 \notin \sigma(A)$ ). Then $A$ is conically averaged if and only if $\left((\mathrm{Id}-A)^{-1}\right)_S$ is positive definite, in which case the formula holds:
$$k(A)=\frac{1}{2} \frac{1}{\lambda_{\min }\left(\left((\Id-A)^{-1}\right)_S\right)}.$$
\end{theorem}

\begin{proof}
Observe that $k(A)>0$. Indeed, if $k(A)=0$, then by Lemma~\ref{zero case} we have $A=\Id$,
so $\Id-A=0$ is not invertible, which contradicts the
assumption. The result follows by
combining $k(T)=1/(2c(\operatorname{Id}-T))$ in Proposition~\ref{kcprop!} and
Proposition~\ref{linear mc}\ref{rho:mono2}.
\end{proof}

A simple example can serve to illustrate Theorem~\ref{nonsingular theorem}.
\begin{example}
Suppose that $X=\mathbb{R}^2$ and consider the matrix
$$
M:=\frac{1}{2}\left[\begin{array}{ll}
1 & 0 \\
1 & 0
\end{array}\right].
$$Then $$
\mathrm{Id}-M=\left[\begin{array}{ll}
\frac{1}{2} & 0 \\
-\frac{1}{2} & 1
\end{array}\right],
$$ which is invertible and $$
(\mathrm{Id}-M)^{-1}=\left[\begin{array}{ll}
2 & 0 \\
1 & 1
\end{array}\right].
$$ Thus, $$
\left((\mathrm{Id}-M)^{-1}\right)_S=\left[\begin{array}{cc}
2 & \frac{1}{2} \\
\frac{1}{2} & 1
\end{array}\right]
$$ and $$
\lambda_{\min }\left(\left((\operatorname{Id}-M)^{-1}\right)_S\right)=\frac{3-\sqrt{2}}{2}.
$$
Therefore,
 $$
k(M)=\frac{1}{2} \frac{1}{\lambda_{\min }\left(\left((\operatorname{Id}-M)^{-1}\right)_S\right)}=\frac{3+\sqrt{2}}{7}.
$$This result agrees with \cite[Example 3.5]{BBM}.
\end{example}

\subsection{General matrix $A$: when $\Fix A\varsupsetneq\{0\}$}\label{s:sub:nonzero}
Observe that $\Fix A\varsupsetneq\{0\}$ if and only if
$\Id-A$ is not invertible.
We now consider how to determine whether a matrix $A$ is conically averaged and compute $k(A)$
when $1 \in \sigma(A)$. A possible approach might be to choose a sequence of matrices $(A_m)_{m\in\NN}$ with $1 \notin \sigma(A_m)$ and $A_m \rightarrow A$, then compute $k(A)$ through $k(A_m)$. Nonetheless, the following example shows that this approach may fail when the sequence is not chosen appropriately.

\begin{example}
$k(\Id) \neq \lim _{\varepsilon \rightarrow 0^{+}}k(\Id+\varepsilon \Id)$.
\end{example}

\begin{proof}
By Example~\ref{e:Id}, we have
 $$
k(\alpha \mathrm{Id})=\begin{cases}
\frac{1-\alpha}{2} & \text{if $\alpha \leq 1$,} \\
+\infty & \text{if $\alpha>1$,}
\end{cases}
$$
which gives $k(\Id)=0$ while $k((1+\varepsilon)\Id)=+\infty$ for every $\varepsilon>0$.
\end{proof}

Before investigating what type of sequence is appropriate, we present the following basic result.

\begin{proposition}\label{usefullinear}
Let $A \in M_n(\mathbb{R})$ be conically averaged. Then the following hold:
\begin{enumerate}
\item \label{i:linear1}
 $k(A)=0$ if and only if $A=\Id$.
\item\label{i:linear2}
$A$ is conically $k(A)$-averaged.
\item\label{i:linear3}
For any $k \in [0, +\infty)$, $A$ is conically $k$-averaged if and only if
\begin{equation}\label{e:linear:char1}
\left(\forall z \in \mathbb{R}^n\right) \quad\|A z\|^2+(1-2 k)\|z\|^2 \leq 2(1-k)\langle z, A z\rangle;
\end{equation}
equivalently,
\begin{equation}\label{e:linear:char2}
\left(\forall z \in \mathbb{R}^n\right) \quad\|z-Az\|^2 \leq 2k(\|z\|^2-\langle z, A z\rangle).
\end{equation}
\end{enumerate}
\end{proposition}

\begin{proof}
\ref{i:linear1} and \ref{i:linear2}: Apply Lemma~\ref{zero case} to the linear operator $A$.

\ref{i:linear3}: When $k>0$, apply the characterization \eqref{e:average2}. When $k=0$, use \ref{i:linear1}.
\end{proof}

The following lemma is crucial, as it provides an \emph{appropriate} class of matrices to compute the modulus of
a conically averaged matrix.

\begin{lemma}\label{singular lemma}
Let $A \in M_n(\mathbb{R})$ be conically averaged. Then $\lim _{\varepsilon \rightarrow 0^+} k(A-\varepsilon \Id)$ exists, and $$k(A)=\lim _{\varepsilon \rightarrow 0^+} k(A-\varepsilon \Id).$$
\end{lemma}

\begin{proof}
By Lemma~\ref{l:minus},
$A-\varepsilon \mathrm{Id}$ is conically $(k(A)+\varepsilon)$-averaged. Thus $k(A-\varepsilon \Id) \leq k(A)+\varepsilon$. Taking limit superior we have
\begin{equation}\label{4.9eq}
\varlimsup_{\varepsilon \rightarrow 0^{+}} k(A-\varepsilon \Id) \leq k(A).
\end{equation}

Applying Proposition~\ref{usefullinear}\ref{i:linear2}\&\ref{i:linear3} to the conically averaged operator $A-\varepsilon \mathrm{Id}$, we have
\begin{equation}\label{e:A-eId}
(\forall z\in\mathbb{R}^n)\ \|z-(A-\varepsilon \Id) z\|^2\leq 2k(A-\varepsilon\Id)(\|z\|^2-\langle z,(A-\varepsilon \Id) z\rangle).
\end{equation}
Note that $\varliminf_{\varepsilon \rightarrow 0^{+}} k(A-\varepsilon \mathrm{Id})$ is finite due to its nonnegativity and \eqref{4.9eq}. Taking $\varliminf$ when $\varepsilon \rightarrow 0^{+}$ in \eqref{e:A-eId}, we obtain
$$
\begin{aligned}
\|z-Az\|^2 & \leq 2\varliminf_{\varepsilon\rightarrow 0^+}[k(A-\varepsilon\Id)(\|z\|^2-\langle z,(A-\varepsilon \Id) z\rangle)] \\& \leq 2\varliminf_{\varepsilon\rightarrow 0^+}k(A-\varepsilon\Id) (\|z\|^2-\langle z,A z\rangle)
\end{aligned}
$$
due to the continuity of norm and inner product and Fact~\ref{f:l:supinf}. It follows from Proposition \ref{usefullinear}\ref{i:linear3} again that $A$ is conically $\varliminf_{\varepsilon \rightarrow 0^{+}}k(A-\varepsilon \mathrm{Id})$-averaged, so
\begin{equation}\label{e:monday}
k(A) \leq \varliminf_{\varepsilon \rightarrow 0^{+}}k(A-\varepsilon \mathrm{Id}).
\end{equation}
Combining \eqref{4.9eq} and \eqref{e:monday} gives
$$
k(A) \leq \varliminf_{\varepsilon \rightarrow 0^{+}}k(A-\varepsilon \mathrm{Id}) \leq \varlimsup_{\varepsilon \rightarrow 0^{+}}k(A-\varepsilon \mathrm{Id}) \leq k(A),
$$ and all inequalities turn into equalities.
\end{proof}

We will generalize Lemma \ref{singular lemma} to a more substantial nonlinear version in Section~\ref{s:nonlinear}.

The singularity of matrices arises only in a discrete manner, which motivates the following definitions.

\begin{definition}\label{d:gammaA}
Let $A \in M_n(\mathbb{R})$.
\begin{enumerate}
\item\label{i:gamma1}
Define $\gamma(A):=\inf \{\lambda \in \sigma(A) \cap \mathbb{R} \mid \lambda>1\}-1$ as the distance between $1$ and the smallest real eigenvalue of $A$ greater than $1$.
\item\label{i:gamma2}
Define the function $\phi_A: (0, \gamma(A)) \rightarrow \RR$ by $$\phi_A(\varepsilon):=\lambda_{\min }\left(\left(((1+\varepsilon) \mathrm{Id}-A)^{-1}\right)_S\right).$$
\end{enumerate}
\end{definition}

\begin{remark}
If $A$ has a real eigenvalue greater than $1$, then the infimum in $\gamma(A)$ can be replaced by minimum since $\sigma(A)$ is a finite set ($A$ has at most $n$ different eigenvalues on the complex plane). If $A$ has no real eigenvalue greater than $1$, then $\gamma(A)=+\infty$ as $\inf \varnothing=+\infty$. Moreover, we have $\gamma(A)>0$ and $(1,1+\gamma(A)) \cap \sigma(A)=\varnothing$. This implies that $(1+\varepsilon) \Id-A$ is invertible for any $0<\varepsilon<\gamma(A)$. Therefore, $\phi_A(\varepsilon)=\lambda_{\min }\left(\left(((1+\varepsilon) \mathrm{Id}-A)^{-1}\right)_S\right)$ is a well-defined real-valued function on an open interval, and we can consider $\lim _{\varepsilon \rightarrow 0^{+}}\phi_A(\varepsilon)$.
\end{remark}

Armed with Theorem~\ref{nonsingular theorem} and Lemma~\ref{singular lemma},
we are now ready to state one of our main results.

\begin{theorem}\emph{\textbf{(conically averaged matrix: characterization I)}}\label{mainresult1}
Let $A \in M_n(\mathbb{R})$. Then $A$ is conically averaged if and only if $\lim _{\varepsilon \rightarrow 0^{+}} \varphi_A(\varepsilon)$ exists and belongs to $(0, +\infty]$, in which case the formula holds: $$
k(A)=\frac{1}{2} \frac{1}{\lim _{\varepsilon \rightarrow 0^+} \varphi_A(\varepsilon)}.
$$
\end{theorem}

\begin{proof}
``$\Rightarrow$":
Suppose $A$ is conically averaged. For any $0<\varepsilon<\gamma(A)$, we have $\Id-(A-\varepsilon\Id)=(1+\varepsilon) \mathrm{Id}-A$ is invertible. Moreover, Lemma~\ref{l:minus} implies that $A-\varepsilon\Id$ is conically averaged since $A$ is conically averaged. Therefore, applying Theorem \ref{nonsingular theorem} we have $$
k(A-\varepsilon \mathrm{Id})=\frac{1}{2} \frac{1}{\lambda_{\min }\left(\left(((1+\varepsilon) \mathrm{Id}-A)^{-1}\right)_S\right)}=\frac{1}{2} \frac{1}{\varphi_A(\varepsilon)} .
$$
Since $A$ is conically averaged, applying Lemma \ref{singular lemma} we have $\lim _{\varepsilon \rightarrow 0^{+}} k(A-\varepsilon \mathrm{Id})$ exists and $$
k(A)=\lim _{\varepsilon \rightarrow 0^{+}} k(A-\varepsilon \mathrm{Id})=\frac{1}{2} \frac{1}{\lim _{\varepsilon \rightarrow 0^{+}} \varphi_A(\varepsilon)}\in[0, +\infty).
$$
Thus $\lim _{\varepsilon \rightarrow 0^{+}} \varphi_A(\varepsilon)$ exists and belongs to $(0, +\infty]$, in which case the formula holds.

``$\Leftarrow$": Suppose $\lim _{\varepsilon \rightarrow 0^{+}} \varphi_A(\varepsilon)$ exists and belongs to $(0, +\infty]$. Then by the property of one-sided limit, there exists $0<\delta<\gamma(A)$ such that for any $\varepsilon \in (0, \delta)$: $\varphi_A(\varepsilon)>0$, i.e., $(((1+\varepsilon) \mathrm{Id}-A)^{-1})_S$ is positive definite. Thus $(1+\varepsilon) \mathrm{Id}-A=\Id-(A-\varepsilon\Id)$ is invertible. Applying Theorem \ref{nonsingular theorem} we have $A-\varepsilon\Id$ is conically averaged for any $\varepsilon \in (0, \delta)$ and $$
k(A-\varepsilon\Id)=\frac{1}{2} \frac{1}{\lambda_{\min }\left(\left(((1+\varepsilon)\operatorname{Id}-A)^{-1}\right)_S\right)}=\frac{1}{2} \frac{1}{\varphi_A(\varepsilon)}.
$$
Since $\lim _{\varepsilon \rightarrow 0^{+}} \varphi_A(\varepsilon) \in (0, +\infty]$, taking limit we have
\begin{equation}\label{e:mod:sat}
\lim _{\varepsilon \rightarrow 0^{+}} k(A-\varepsilon \mathrm{Id})=\frac{1}{2} \frac{1}{\lim _{\varepsilon \rightarrow 0^{+}} \varphi_A(\varepsilon)}\in[0,+\infty).
\end{equation}
Applying Proposition \ref{usefullinear}\ref{i:linear3} to the conically averaged operator $A-\varepsilon \mathrm{Id}$, we have
$$(\forall z \in \mathbb{R}^n)\
\|(A-\varepsilon \Id) z\|^2+\left(1-2 k(A-\varepsilon \Id)\right)\|z\|^2 \leq 2(1-k(A-\varepsilon \Id))\langle z,(A-\varepsilon \Id) z\rangle.
$$
When $\varepsilon \rightarrow 0^{+}$, we obtain
\begin{equation}\label{e:linear:one}
(\forall z \in \mathbb{R}^n)\ \|A z\|^2+(1-2 \lim _{\varepsilon \rightarrow 0^{+}} k(A-\varepsilon \mathrm{Id}))\|z\|^2 \leq 2(1-\lim _{\varepsilon \rightarrow 0^{+}} k(A-\varepsilon \mathrm{Id}))\langle z,A z\rangle
\end{equation}
due to the continuity of norm and inner product.
Note that we already proved $\lim_{\varepsilon \rightarrow 0^{+}} k(A-\varepsilon \mathrm{Id}) \in[0,+\infty)$
by \eqref{e:mod:sat}, and that \eqref{e:linear:one} holds. Hence $A$ is conically $\lim_{\varepsilon \rightarrow 0^{+}}k(A-\varepsilon \mathrm{Id})$-averaged
by Proposition \ref{usefullinear}\ref{i:linear3} again.
Altogether, we complete the proof.
\end{proof}

\begin{corollary}\emph{\textbf{(matrix with nonzero fixed point: formula I)}}\label{singular theorem}
Let $A \in M_n(\mathbb{R})$ and suppose that $\mathrm{Id}-A$ is not invertible (i.e., $1\in \sigma(A)$). Then $A$ is conically averaged if and only if $\lim _{\varepsilon \rightarrow 0^+} \lambda_{\min }\left(\left(((1+\varepsilon) \operatorname{Id}-A)^{-1}\right)_S\right)$ exists and belongs to $(0, +\infty]$, in which case the formula holds: $$
k(A)=\frac{1}{2} \frac{1}{\lim _{\varepsilon \rightarrow 0^+} \lambda_{\min }\left(\left(((1+\varepsilon) \mathrm{Id}-A)^{-1}\right)_S\right)}.
$$
\end{corollary}

The following example illustrates how to algorithmically determine the modulus of a square matrix
using Corollary~\ref{singular theorem}.

\begin{example}\label{example PVPU}
Let $$
M:=\left(\begin{array}{ccc}
1 & 0 & 0 \\
0 & \frac{1}{2} & 0 \\
0 & \frac{1}{2} & 0
\end{array}\right).
$$ Then $M \notin \mathbb{S}^n$ and $1 \in \sigma(M)$. Thus both
Proposition~\ref{kSn} and Theorem~\ref{nonsingular theorem} are not
applicable to compute $k(M)$. We have for every $\varepsilon>0$,
$$
\begin{aligned}
((1+\varepsilon) \Id-M)^{-1} & =\left(\begin{array}{ccc}
\varepsilon & 0 & 0 \\
0 & \frac{1}{2}+\varepsilon & 0 \\
0 & -\frac{1}{2} & 1+\varepsilon
\end{array}\right)^{-1} \\
& =\left(\begin{array}{ccc}
\frac{1}{\varepsilon} & 0 & 0 \\
0 & \frac{2}{2 \varepsilon+1} & 0 \\
0 & \frac{1}{2 \varepsilon^2+3 \varepsilon+1} & \frac{1}{\varepsilon+1}
\end{array}\right),
\end{aligned}
$$ thus
$$
\left(((1+\varepsilon) \operatorname{Id}-M)^{-1}\right)_S=\left(\begin{array}{ccc}
\frac{1}{\varepsilon} & 0 & 0 \\
0 & \frac{2}{2 \varepsilon+1} & \frac{1}{4 \varepsilon^2+6 \varepsilon+2} \\
0 & \frac{1}{4 \varepsilon^2+6 \varepsilon+2} & \frac{1}{\varepsilon+1}
\end{array}\right),
$$and
$$\sigma\left(\left(((1+\varepsilon) \mathrm{Id}-M)^{-1}\right)_S\right)=\left\{
\frac{1}{\varepsilon},  \frac{4 \varepsilon-\sqrt{2}+3}{4 \varepsilon^2+6 \varepsilon+2},  \frac{4 \varepsilon+\sqrt{2}+3}{4 \varepsilon^2+6 \varepsilon+2}\right
\}.$$
Therefore,
$$\phi_M(\varepsilon)=\lambda_{\min}\left(\left(((1+\varepsilon) \mathrm{Id}-M)^{-1}\right)_S\right)=\frac{4 \varepsilon-\sqrt{2}+3}{4 \varepsilon^2+6 \varepsilon+2}.$$
Taking the limit as $\varepsilon\rightarrow 0^+$ yields
$$\lim _{\varepsilon \rightarrow 0^{+}} \lambda_{\min }\left(\left(((1+\varepsilon) \mathrm{Id}-M)^{-1}\right)_S\right)=\lim _{\varepsilon \rightarrow 0^{+}} \frac{4 \varepsilon-\sqrt{2}+3}{4 \varepsilon^2+6 \varepsilon+2}=\frac{3-\sqrt{2}}{2}>0.$$
Thus, by Corollary \ref{singular theorem}, we conclude that $M$ is conically averaged and $$
k(M)=\frac{1}{2 \lim _{\varepsilon \rightarrow 0^{+}} \lambda_{\min }\left(\left(((1+\varepsilon) \mathrm{Id}-M)^{-1}\right)_S\right)}=\frac{1}{3-\sqrt{2}}=\frac{3+\sqrt{2}}{7}.
$$
\end{example}

Given $A \in M_n(\mathbb{R})$, we see from the example that finding the function $\varphi_A$ is crucial for determining $k(A)$. It would be problematic if $\varphi_A$ is badly discontinuous. Fortunately, there is a remarkable fact regarding the continuity of eigenvalues, which ensures that such pathologies do not occur in this context.

\begin{fact}\emph{(\textbf{Kato})}\emph{\cite[Theorem 5.2]{KT}(see also \cite[Theorem 3]{LZ})}\label{Katofact}
Suppose that $D \subset \mathbb{C}$ is a connected domain and that $A: D \rightarrow M_n(\mathbb{C})$ is a continuous function. If (1) $D$ is a real interval, or (2) $A(t)$ has only real eigenvalues, then there exist $n$ eigenvalues (counted with algebraic multiplicities) of $A(t)$ that can be parameterized as continuous functions $\lambda_1(t), \ldots, \lambda_n(t)$ from $D$ to $\mathbb{C}$. In the second case, one can set $\lambda_1(t) \geq \cdots \geq \lambda_n(t)$.
\end{fact}

\begin{proposition}
Let $A \in M_n(\mathbb{R})$. Then $\varphi_A$ is continuous on $(0,\gamma(A))$, where $\gamma(A)$ is given in Definition~\ref{d:gammaA}\ref{i:gamma1}.
\end{proposition}

\begin{proof}
The mapping $\left(((1+\cdot) \mathrm{Id}-A)^{-1}\right)_S$ from the interval $(0, \gamma(A))$ to $M_n(\mathbb{R})$ is continuous since $(\cdot)^{-1}$ and $(\cdot)_S$ are continuous operations over $\left(M_n(\mathbb{R}),\|\cdot\|_2\right)$. On the other hand, the domain $(0, \gamma(A))$ is connected and $\left(((1+\varepsilon) \mathrm{Id}-A)^{-1}\right)_S$ has only real eigenvalues for every $\varepsilon \in (0, \gamma(A))$.
Therefore, $\varphi_A(\varepsilon)=\lambda_{\min }\left(\left(((1+\varepsilon) \operatorname{Id}-A)^{-1}\right)_S\right)$ is continuous on $(0,\gamma(A))$ by Fact \ref{Katofact}.
\end{proof}

\subsection{Continuity of the modulus of conical averagedness function}

Let $\kappa: M_n(\mathbb{R}) \rightarrow[0,+\infty]$ be the restriction of the modulus on $M_n(\mathbb{R})$, i.e., $\kappa(A):=k(A)$ for every $A \in M_n(\mathbb{R})$. Then $\kappa$ is an extended nonnegative-valued function on $M_n(\mathbb{R})$ and $\operatorname{dom} \kappa:=\{A \in M_n(\mathbb{R}) \mid \kappa(A) < +\infty\}$ is the set of all conically averaged square matrices.

\begin{theorem}\label{modlinearlowersemi}
The function $\kappa$ is lower semicontinuous on $M_n(\mathbb{R})$.
\end{theorem}

\begin{proof}
Let $(A_m)_{m \in \mathbb{N}}$ be a sequence in $M_n(\mathbb{R})$ such that $A_m\rightarrow A$.
We aim to show that $\kappa(A) \leq \varliminf_{m \rightarrow \infty} \kappa\left(A_m\right)$. If $\varliminf_{m\rightarrow \infty} \kappa(A_m)=+\infty$, then $\kappa(A) \leq \varliminf_{m\rightarrow \infty} \kappa(A_m)$. If $\varliminf_{m \rightarrow \infty} \kappa(A_m) \in [0, +\infty)$, then there exists a subsequence $(m_j)_{j \in \mathbb{N}}$ such that $\kappa(A_{m_j}) \rightarrow \varliminf_{m \rightarrow \infty} \kappa\left(A_m\right)$ when $j \rightarrow \infty$. Thus there exists $N>0$ such that for any $j>N$: $\kappa(A_{m_j})\in [0, +\infty)$. Applying Proposition \ref{usefullinear}\ref{i:linear3} to the conically averaged operator $A_{m_j}$, we have
$$
(\forall z \in \mathbb{R}^n)\ \|A_{m_j} z\|^2+(1-2 \kappa(A_{m_j}))\|z\|^2 \leq 2(1-\kappa(A_{m_j}))\langle z,A_{m_j} z\rangle.
$$
Note that $A_{m_j} \rightarrow A$ implies $A_{m_j}z \rightarrow Az$. Letting $j \rightarrow \infty$ yields
\begin{equation}\label{e:sunday1}
(\forall z \in \mathbb{R}^n)\ \|A z\|^2+(1-2 \varliminf_{m \rightarrow \infty} \kappa(A_m))\|z\|^2 \leq 2(1-\varliminf_{m \rightarrow \infty} \kappa(A_m))\langle z, A z\rangle
\end{equation}
due to the continuity of norm and inner product. Because of $\varliminf_{m \rightarrow \infty} \kappa\left(A_m\right) \in[0,+\infty)$ and \eqref{e:sunday1}, we derive that $A$ is conically $\varliminf_{m \rightarrow \infty} \kappa\left(A_m\right)$-averaged by applying Proposition \ref{usefullinear}\ref{i:linear3} again. Therefore,
$
\kappa(A) \leq \varliminf_{m \rightarrow \infty} \kappa\left(A_m\right).
$
\end{proof}

\begin{corollary}\label{c:klsc}
The function $\kappa$ is proper, lower semicontinuous and convex on $M_n(\mathbb{R})$, but not strictly convex.
Moreover, $\argmin \kappa=\{\Id\}$ and $\min\kappa=0$.
\end{corollary}

\begin{proof}
Since $\kappa\geq 0$ and $\kappa(\Id)=0$ by Proposition~\ref{usefullinear}\ref{i:linear1}, we have that
$\kappa$ is proper, $\min\kappa=0$ with $\argmin\kappa=\{\Id\}$.
While Theorem \ref{modlinearlowersemi} gives the lower semicontinuity of $\kappa$, Proposition~\ref{newOYprop} yields
the convexity. To see that $\kappa$ is not strictly convex, we use
$\kappa((1-\lambda)\Id)=\lambda/2$ for $\lambda\geq 0$ by Example~\ref{e:Id}.
Altogether the proof is complete.
\end{proof}

The following result concerns the continuity of convex functions. Below $\inte\dom f$ denotes the interior of $\dom f$, and $\mbox{cont} f$ is the set of points at which $f$ is continuous.
\begin{fact}\emph{\cite[Corollary 8.39]{BC}}\label{klscfact}
Let $f: X \rightarrow (-\infty,+\infty]$ be proper and convex, and suppose that one of the following holds:
\begin{enumerate}
\item $f$ is bounded above on some neighborhood.
\item $f$ is lower semicontinuous.
\item $X$ is finite-dimensional.
\end{enumerate}
Then $\operatorname{cont} f=\operatorname{int} \operatorname{dom} f$.
\end{fact}

Let $\mathcal{C}$ denote the set of all conically averaged matrices in $M_n(\mathbb{R})$. Then clearly we have $\operatorname{dom} \kappa=\mathcal{C}$.

\begin{corollary}\label{c:set:conical}
The function $\kappa$ is continuous on $\operatorname{int} \mathcal{C}$ and $\operatorname{cont} \kappa=\operatorname{int} \mathcal{C}$. In particular,
$$\BB(0,1):=\menge{A\in M_n(\mathbb{R})}{\|A\|_{2}<1}\subset \operatorname{int} \mathcal{C}.$$
\end{corollary}

\begin{proof}
Combine Corollary \ref{c:klsc} and Fact \ref{klscfact}.
When $A\in \BB(0,1)$, the mapping $A:\RR^n\rightarrow\RR^n$ is a Banach contraction. By \cite[Example 3.2]{nollphan} every Banach contraction $A$ is averaged with $\kappa(A)<1$, in particular, conically averaged. Then $\BB(0,1)\subset\inte \mathcal{C}$.
\end{proof}

It is natural to ask whether one can characterize $\inte \mathcal{C}$ in Corollary~\ref{c:set:conical}. To this end, we first refine the characterization of $\mathcal{C}$. In what follows,
$\succeq$ denotes the Loewner order, $\ker(\cdot)$ denotes the kernel of a matrix, and $(\cdot)^\perp$ denotes the orthogonal complement of a set.

\begin{theorem}\emph{\textbf{(conically averaged matrix: characterization II)}}\label{t:CharacterC}
Let $A\in M_n(\mathbb R)$. Then $A$ is conically averaged if and only if $\mathrm{Id}-A_S \succeq 0$ and $\operatorname{ker}\left(\mathrm{Id}-A_S\right) \subseteq \operatorname{ker}(\mathrm{Id}-A)$.
\end{theorem}

\begin{proof}
Let $A\in M_n(\mathbb R)$, and set $B:=\Id-A$. By Proposition~\ref{usefullinear}\ref{i:linear3}, $A$ is conically
$k$-averaged with $k\ge 0$ if and only if
\begin{equation}\label{eq:2kB}
\left(\forall z\in\mathbb R^n\right)\ \|Bz\|^2\leq 2k\langle z,Bz\rangle.
\end{equation}
We claim that \eqref{eq:2kB} holds if and only if
\begin{equation}\label{eq:kerB}
B_S\succeq 0
\quad\text{and}\quad
\ker B_S\subseteq \ker B.
\end{equation}
Indeed, suppose first that \eqref{eq:2kB} holds. If $k=0$, then $B=0$, hence $B_S=0$, and \eqref{eq:kerB} follows.
If $k>0$, then by $\langle z,Bz\rangle=\langle z,B_Sz\rangle$, \eqref{eq:2kB} turns into $$
0\le \|B z\|^2 \leq 2 k\langle z, B z\rangle=2k\langle z,B_Sz\rangle  \quad\left(\forall z \in \mathbb{R}^n\right),$$
hence
$B_S\succeq 0$. Moreover, if $z\in\ker B_S$, then the same inequality gives
$$
\|Bz\|^2\leq 2k\langle z,B_Sz\rangle=0,
$$
hence $Bz=0$. Thus $\ker B_S\subseteq\ker B$.

Conversely, assume \eqref{eq:kerB}. If $B_S=0$, then $\ker B_S=\mathbb R^n$, so $\operatorname{ker} B=\mathbb R^n$, which implies $B=0$. Thus \eqref{eq:2kB} holds
with $k=0$. If $B_S\neq 0$, let $z\in \mathbb R^n$, and decompose
\begin{equation}\label{eq:decom}
z=u+v,
\qquad
u\in\ker B_S,\quad v\in(\ker B_S)^\perp .
\end{equation}
Since $\ker B_S\subseteq\ker B$, we have $Bz=Bv$. Also, because $B_S$ is symmetric and
$B_Su=0$,
$$
\langle z,B_Sz\rangle=\langle v,B_Sv\rangle .
$$
Moreover, since $B_S \succeq 0$, the restriction of $B_S$ to $(\ker B_S)^\perp$ is positive definite, hence
$$
(\forall v\in(\ker B_S)^\perp)\ \langle v,B_Sv\rangle\geq \lambda_+\|v\|^2,
$$
where $$
\lambda_{+}:=\lambda_{\min }\left(\left.B_S\right|_{(\operatorname{ker} B_S)^{\perp}}\right)>0.
$$
Therefore,
$$
\|Bz\|^2
=
\|Bv\|^2
\leq
\|B\|_2^2\|v\|^2
\leq
\frac{\|B\|_2^2}{\lambda_+}\langle v,B_Sv\rangle
=
\frac{\|B\|_2^2}{\lambda_+}\langle z,B_Sz\rangle .
$$
Thus \eqref{eq:2kB} holds with
$$
k=\frac{\|B\|_2^2}{2\lambda_+}.
$$
This completes the proof.
\end{proof}

\begin{remark}
Theorem~\ref{t:CharacterC} immediately recovers the conical averagedness criterion for symmetric matrices, for which $A_S=A$. It also implies that
skew-symmetric matrices, for which $A_S=0$, are always conically averaged; see Propositions~\ref{kSn} and \ref{p:skew}.
\end{remark}

\begin{remark}\label{r:equ} 
In Theorem~\ref{t:CharacterC}, the kernel condition
$$
\ker(\Id-A_S)\subseteq \ker(\Id-A)
$$
can  be equivalently replaced by
$\ker(\Id-A_S)=\ker(\Id-A)$, i.e.,
$\Fix A_S=\Fix A.$
Indeed, since $\Id-A_S\succeq 0$, we always have
$$
\ker(\Id-A)\subseteq \ker(\Id-A_S).
$$
To see this, let $x\in \ker(\Id-A)$. Then
$$
\langle (\Id-A_S)x,x\rangle
=
\langle (\Id-A)x,x\rangle
=
0.
$$
Since $\Id-A_S\succeq 0$, this implies $(\Id-A_S)x=0$, and hence
$x\in \ker(\Id-A_S)$.
\end{remark}

For verifying whether a matrix is conically averaged, the following corollary provides a nicely computable criterion. It avoids the
eigenvalue limit condition in Theorem~\ref{mainresult1}, though this
criterion does not yield the value of $k(A)$, whereas Theorem~\ref{mainresult1} does. Below, $\rank(\cdot)$ denotes the rank of a matrix.

\begin{corollary}\label{c:computechar}
Let $A\in M_n(\mathbb R)$. Then $A$ is conically averaged if and only if $\lambda_{\max}(A_S) \leq 1$ and $\operatorname{rank}\left(\mathrm{Id}-A_S\right)=\operatorname{rank}\left(\mathrm{Id}-A\right)$.
\end{corollary}

\begin{proof}
Note that
$$\mathrm{Id}-A_S \succeq 0 \Longleftrightarrow \lambda_{\max}(A_S) \leq 1.$$
By Theorem \ref{t:CharacterC} and Remark \ref{r:equ}, it suffices to show that under the condition $\mathrm{Id}-A_S\succeq 0$, we have $$\operatorname{ker}\left(\mathrm{Id}-A_S\right)=\operatorname{ker}(\mathrm{Id}-A) \Longleftrightarrow\operatorname{rank}\left(\mathrm{Id}-A_S\right)=\operatorname{rank}\left(\mathrm{Id}-A\right).$$
The ``$\Rightarrow$" direction follows by the rank plus nullity theorem, see, e.g., \cite[equation (4.4.15), page 199]{meyer}. Conversely, suppose that
$$
\operatorname{rank}\left(\mathrm{Id}-A_S\right)=\operatorname{rank}\left(\mathrm{Id}-A\right).
$$
Using the same argument in Remark \ref{r:equ}, we have
$\operatorname{ker}(\operatorname{Id}-A) \subseteq \operatorname{ker}\left(\operatorname{Id}-A_S\right)$ since $\mathrm{Id}-A_S \succeq 0$. Again by the rank plus nullity theorem, we obtain
$$
\dim\ker (\mathrm{Id}-A_S)=\dim\ker (\mathrm{Id}-A),
$$
which forces
$
\operatorname{ker}(\operatorname{Id}-A) =\operatorname{ker}\left(\operatorname{Id}-A_S\right).
$
\end{proof}

\begin{lemma}[Weyl's inequality]\emph{\cite[Theorem~4.3.1]{HJ}}
Let $H,K\in \mathbb{S}^n$. Then
$$
\lambda_{\min}(H+K)
\geq
\lambda_{\min}(H)+\lambda_{\min}(K).
$$
\end{lemma}

We are now ready to derive the following amazing explicit form of $\operatorname{int}\mathcal{C}$.

\begin{theorem}\label{t:interior:done}
We have $$
\operatorname{int} \mathcal{C} =\menge{A\in M_n(\mathbb R)}{\lambda_{\max}(A_S)<1}.
$$
Consequently, if $\lambda_{\max }\left(A_S\right)<1$, then $A$ is conically averaged.
\end{theorem}

\begin{proof}
Suppose $\lambda_{\max }\left(A_S\right)<1$, equivalently, $\lambda_{\min }\left(\operatorname{Id}-A_S\right)>0$. Then $\mathrm{Id}-A_S$ is invertible, so $\operatorname{ker}\left(\operatorname{Id}-A_S\right)=\{0\}$. Hence the kernel condition in Theorem \ref{t:CharacterC} is automatic, giving $A\in \mathcal{C}$. Let $\delta:=\lambda_{\min }\left(\operatorname{Id}-A_S\right)>0$ and let $E \in M_n(\mathbb{R})$ with $\|E\|_2<\delta$. Then $$
\|E_S\|_2
=
\left\|\frac{E+E^\top}{2}\right\|_2
\leq
\frac{1}{2}\bigl(\|E\|_2+\|E^\top\|_2\bigr)
=
\frac{1}{2}\bigl(\|E\|_2+\|E\|_2\bigr)
=
\|E\|_2.
$$
By Weyl's inequality,
$$\begin{aligned}
\lambda_{\min }\left(\operatorname{Id}-(A+E)_S\right) &=\lambda_{\min }\left(\operatorname{Id}-A_S-E_S\right) \\& \geq \lambda_{\min }\left(\operatorname{Id}-A_S\right)+\lambda_{\min }\left(-E_S\right) \\&=\lambda_{\min }\left(\operatorname{Id}-A_S\right)-\lambda_{\max }\left(E_S\right)
\\&\geq \lambda_{\min }\left(\operatorname{Id}-A_S\right)-\left\|E_S\right\|_2>0.
\end{aligned}$$
Therefore $\lambda_{\max }\left((A+E)_S\right)<1$. In view of Theorem~\ref{t:CharacterC},
this implies $A+E\in \mathcal{C}$. Thus $A\in \operatorname{int} \mathcal{C}$.

Conversely, suppose $A\in \inte\mathcal{C}$. Then $\Id-A_{S}\succeq 0$, i.e., $\lambda_{\max}(A_S)\leq 1$, by Theorem~\ref{t:CharacterC}.
We must show $\lambda_{\max}(A_S)<1$.
Suppose that $\Id-A_S$ is singular. Choose
$0\neq u\in\ker(\Id-A_S)$. For $t>0$, define
$
A_t:=A+tuu^\top .
$
We have
$$
\Id-(A_t)_S=\Id-A_S-tuu^\top,
$$
so
$$
\langle u,(\Id-(A_t)_S)u\rangle
=\langle u,-t u u^{\top}u\rangle=
-t\|u\|^4<0.
$$
Hence $\Id-(A_t)_S\not\succeq 0$, so $A_t\notin \mathcal{C}$ by Theorem \ref{t:CharacterC}. But $A_t\to A\in \mathcal{C}$ as $t\rightarrow0^+$. Thus $A\notin\operatorname{int}\mathcal{C}$, which is a
contradiction. This completes the proof.
\end{proof}

\begin{remark} Theorem~\ref{t:interior:done} implies that
if $A\in \inte \mathcal{C}$, then
$
k(A)=\lim_{\varepsilon\to 0^+} k(A-\varepsilon \Id).
$
Indeed, for every $\varepsilon>0$,
$$
\lambda_{\max}\big((A-\varepsilon \Id)_S\big)
=
\lambda_{\max}(A_S)-\varepsilon
<1.
$$
Hence $A-\varepsilon \Id\in \inte \mathcal{C}$, and the assertion follows
from the continuity of $k$ on $\inte \mathcal{C}$. This argument, however,
does not recover Lemma~\ref{singular lemma} in full generality, since in
that lemma one may have $A\in \mathcal{C}\setminus \inte \mathcal{C}$.
\end{remark}

\section{Angle between two subspaces}\label{s:angleb}

In this section, we investigate the Dixmier and Friedrichs angles between two subspaces.
As an application of our main results on linear conically averaged mappings,
we derive new formulae to compute both
angles. Recall that the cosine of the \emph{Dixmier angle} \cite{DJ} (also known as the minimal angle)
of two closed linear subspaces $U, V$ of a Hilbert space $X$ is
$$
c_D(U, V):=\sup \{\langle u, v\rangle \mid u \in U, v \in V,\|u\| \leq 1,\|v\| \leq 1\},
$$
and the cosine of the \emph{Friedrichs angle} \cite{FK} is
$$
c_F(U, V):=c_D(U \cap(U \cap V)^{\perp}, V \cap(U \cap V)^{\perp}).
$$
Note that $c_D(U, V)=c_F(U, V)$ if $U \cap V=\{0\}$. These angles are significant in describing convergence rates for projection methods such as the cyclic projection algorithm, the Douglas-Rachford algorithm for subspaces, etc. For more information on this topic,  see, e.g., \cite{bbnpw, DFpaper, DFbook, DH}. Below, $P_{U}, P_{V}$ denote the projection mappings onto subspaces $U, V$, respectively.

\subsection{Dixmier angle}

Our formula on computing the Dixmier angle relies on two important facts.
The first one connects the modulus of composition of projections onto subspaces
with the cosine of the Friedrichs angle.

\begin{fact}\emph{\cite[Corollary 3.3]{BBM}}\label{angle fact}
Let $U, V$ be closed linear subspaces of $X$. Then either $U=V=X$ and $k\left(P_V P_U\right)=0$, or
$$
k\left(P_V P_U\right)=\frac{1+c_F(U, V)}{2+c_F(U, V)}.
$$
\end{fact}
The next one concerns the fixed point set of compositions of averaged mappings.

\begin{fact}\emph{\cite[Corollary 4.51]{BC}}\label{Fix PVPU}
Let $T_1, \ldots, T_m$ be averaged operators on $X$ such that $C:=\bigcap_{i=1}^m\Fix T_i \neq \varnothing$. Then $\operatorname{Fix}\left(T_m \cdots T_2 T_1\right)=C$.
\end{fact}

\begin{lemma}\label{angle lemma}
Let $U, V$ be closed linear subspaces of $X$, and let $x \in X$. Then $P_V P_U x=x$ if and only if $x \in U \cap V$.
\end{lemma}

\begin{proof}
Since $P_V$ and $P_U$ are $(1/2)$-averaged and $(\operatorname{Fix}P_V) \cap (\operatorname{Fix}P_U)=U \cap V \neq \varnothing$, we have $\operatorname{Fix}P_V P_U=\left(\operatorname{Fix} P_V\right) \cap\left(\operatorname{Fix} P_U\right)=U \cap V$ by Fact \ref{Fix PVPU}.
\end{proof}

When $U \cap V \neq \{0\}$, the Dixmier angle is known to be trivial, as $c_D(U, V)=1$ by applying the Cauchy-Schwarz inequality.  Now we propose an explicit formula for the Dixmier angle when $U \cap V=\{0\}$,
which depends on Theorem \ref{nonsingular theorem}.

\begin{theorem}\emph{\textbf{(Dixmier angle formula)}}\label{angle theorem1}
Let $U, V$ be linear subspaces of $\RR ^n$. Then the following hold:
\begin{enumerate}
\item\label{i:dixmier1}
If $U \cap V \neq \{0\}$, then $c_D(U, V)=1$.
\item\label{i:dixmier2}
If $U \cap V = \{0\}$, then $$
c_D(U, V)=\frac{1}{2 \lambda_{\min }\left(\left(\left(\mathrm{Id}-P_V P_U\right)^{-1}\right)_S\right)-1}-1.
$$
\end{enumerate}
\end{theorem}

\begin{proof}
\ref{i:dixmier1}:
If $U \cap V \neq \{0\}$, the result follows by taking the unit vector in $U \cap V$ and applying the Cauchy-Schwarz inequality.

\ref{i:dixmier2}:
If $U \cap V=\{0\}$, we have $1 \notin \sigma(P_V P_U)$ by Lemma \ref{angle lemma}.
Thus, by Theorem~\ref{nonsingular theorem},
\begin{equation}\label{e:pvpu}
k(P_V P_U)=\frac{1}{2} \frac{1}{\lambda_{\min }\left(\left((\mathrm{Id}-P_V P_U)^{-1}\right)_S\right)}.
\end{equation}
Also, by Fact~\ref{angle fact},
we have $$k\left(P_V P_U\right)=\frac{1+c_F(U, V)}{2+c_F(U, V)},$$
 equivalently
\begin{equation}\label{e:angle}
c_{F}(U,V)=\frac{1}{\displaystyle \frac{1}{k(P_{V}P_{U})}-1}-1.
\end{equation}
Since $c_D(U, V)=c_F(U, V)$ when $U \cap V=\{0\}$, the result follows by combining equations \eqref{e:pvpu} and
\eqref{e:angle}.
\end{proof}

The following example illustrates Theorem \ref{angle theorem1}.
\begin{example}
In $\mathbb{R}^2$, let $V$ be the $x$-axis and $U$ be the line with slope $\tan \theta$, and $\theta \in (0,\frac{\pi}{2})$. Then $$P_V=\left(\begin{array}{ll}1 & 0 \\ 0 & 0\end{array}\right), \text{ and }
P_U=\left(\begin{array}{cc}(\cos \theta)^2 & \cos \theta \sin \theta \\ \sin \theta \cos \theta & (\sin \theta)^2\end{array}\right).$$
We have $$P_V P_U=\left(\begin{array}{cc}(\cos \theta)^2 & \cos \theta \sin \theta \\ 0 & 0\end{array}\right),$$ and $$\Id-P_V P_U=\left(\begin{array}{cc}(\sin \theta)^2 & -\cos \theta \sin \theta \\ 0 & 1\end{array}\right).$$
Thus, $$(\mathrm{Id}-P_V P_U)^{-1}=\left(\begin{array}{cc}
\frac{1}{(\sin \theta)^2} & \frac{\cos \theta}{\sin \theta} \\
0 & 1
\end{array}\right),$$ and $$(\left(\mathrm{Id}-P_V P_U\right)^{-1})_S=\left(\begin{array}{cc}
\frac{1}{(\sin \theta)^2} & \frac{\cos \theta}{2 \sin \theta} \\
\frac{\cos \theta}{2 \sin \theta} & 1
\end{array}\right).$$
Therefore, $$\lambda_{\min }\left(\left(\left(\mathrm{Id}-P_V P_U\right)^{-1}\right)_S\right)=\frac{1+(\sin \theta)^2-\cos \theta}{2 (\sin \theta)^2},$$ and $$
\begin{aligned}
c_D(U, V) & =\frac{1}{2  \lambda_{\min }\left(\left(\left((\Id-P_V P_U\right)^{-1}\right)_S\right)-1}-1 \\
& =\frac{(\sin \theta)^2}{1+(\sin \theta)^2-\cos \theta-(\sin \theta)^2}-1 \\
& =\cos \theta,
\end{aligned}
$$
which implies that the Dixmier angle between $U$ and $V$ is $\theta$. See Figure~\ref{f:angled} below.
\end{example}

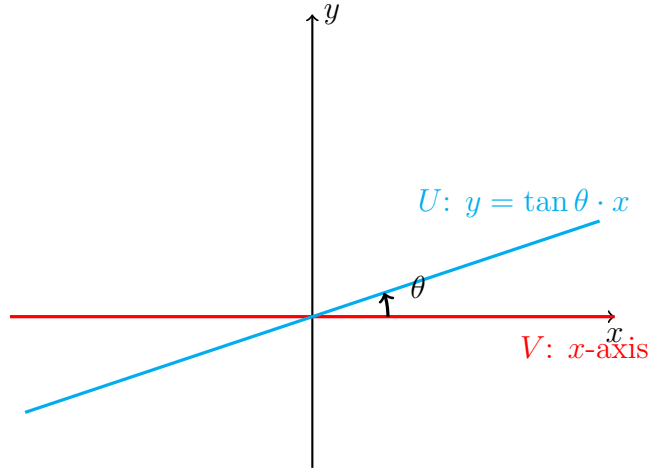
\begin{figure}[H]
\centering
\begin{tikzpicture}[scale=2]

\draw[thick,->] (-2,0) -- (2,0) node[anchor=north]{$x$};
\draw[thick,->] (0,-1) -- (0,2) node[anchor=west]{$y$};

\draw[red, very thick] (-2,0) -- (2,0);
\node[red] at (1.8,-0.2) {$V$: $x$-axis};

\draw[cyan,very thick] (-1.9,-0.633) -- (1.9,0.633);
\node[cyan] at (1.4,0.75) {$U$: $y = \tan\theta \cdot x$};

\draw[->,very thick] (0.5,0) arc[start angle=0,end angle=18.43,radius=0.5];
\node at (0.7,0.2) {$\theta$};

\end{tikzpicture}
\caption{Dixmier angle between two lines}
  \label{f:angled}
\end{figure}

\subsection{Friedrichs angle}

The Friedrichs angle can be viewed as a generalization of the Dixmier angle in the singular case, i.e., when $U \cap V \neq \{0\}$. For example, consider two planes that intersect along a line. Then the Dixmier angle between them is trivial, while the Friedrichs angle represents the nontrivial angle in the usual geometric sense.

As one of our main results and an amazing application of Theorem \ref{mainresult1}, we present a formula for computing the Friedrichs angle.

\begin{theorem}\emph{\textbf{(Friedrichs angle: formula I)}}\label{angle theorem2}
Let $U, V$ be linear subspaces of $\RR ^n$. Then the following hold:
\begin{enumerate}
\item \label{i:fried1}
If $U \cap V = \mathbb{R}^n$ (i.e., $U=V=\mathbb{R}^n$), then $c_F(U, V)=0$.
\item \label{i:fried2}
If $U \cap V=\{0\}$, then
$$
c_F(U, V)=\frac{1}{2 \lambda_{\min }\left(\left(\left(\mathrm{Id}-P_V P_U\right)^{-1}\right)_S\right)-1}-1.
$$
\item\label{i:fried3}
 If $U \cap V \neq \{0\}$ and $U \cap V \neq \mathbb{R}^n$, then$$
c_F(U, V)=\frac{1}{2 \lim _{\varepsilon \rightarrow 0^{+}} \lambda_{\min }\left(\left(((1+\varepsilon) \mathrm{Id}-P_V P_U)^{-1}\right)_S\right)-1}-1.
$$
\end{enumerate}
\end{theorem}

\begin{proof}
\ref{i:fried1}: This follows by $(U \cap V)^{\perp}=\{0\}$.

\ref{i:fried2}: This follows from Theorem \ref{angle theorem1}, since $c_F(U, V)=c_D(U, V)$ if $U \cap V=\{0\}$.

\ref{i:fried3}: If $U \cap V \neq \{0\}$, we have
$1 \in \sigma\left(P_V P_U\right)$ by Lemma \ref{angle lemma}.
Thus, by Corollary \ref{singular theorem},
\begin{equation}\label{e:mon1}
k(P_V P_U)=\frac{1}{2} \frac{1}{\lim _{\varepsilon \rightarrow 0^{+}} \lambda_{\min }\left(\left(((1+\varepsilon) \mathrm{Id}-P_V P_U)^{-1}\right)_S\right)}.
\end{equation}
Since $U \cap V \neq \mathbb{R}^n$, we have
\begin{equation}
k\left(P_V P_U\right)=\frac{1+c_F(U, V)}{2+c_F(U, V)}
\end{equation}
 by Fact \ref{angle fact},
  equivalently
\begin{equation}\label{e:mon2}
c_{F}(U,V)=\frac{1}{\displaystyle \frac{1}{k(P_{V}P_{U})}-1}-1.
\end{equation}
The result then follows by combining equations \eqref{e:mon1} and \eqref{e:mon2}.
\end{proof}

The next example illustrates how to algorithmically compute the Friedrichs angle.
\begin{example}
In $\mathbb{R}^3$, let $$U:=\menge{(x,y,0)}{x,y\in\RR}, \text{ and }
 V:=\menge{(x,y,y)}{x,y\in\RR}.$$
 Then $U$ and $V$ are two planes intersecting along a line such that
$$P_U=\left(\begin{array}{lll}1 & 0 & 0 \\ 0 & 1 & 0 \\ 0 & 0 & 0\end{array}\right) \text{ and } P_V=\left(\begin{array}{lll}1 & 0 & 0 \\ 0 & \frac{1}{2} & \frac{1}{2} \\ 0 & \frac{1}{2} & \frac{1}{2}\end{array}\right).$$
(If the subspaces are given by the sets of linearly independent vectors, our first step would be to compute the projection matrices.) We have $$P_V P_U=\left(\begin{array}{lll}
1 & 0 & 0 \\
0 & \frac{1}{2} & 0 \\
0 & \frac{1}{2} & 0
\end{array}\right),$$ which coincides with the matrix $M$ in Example \ref{example PVPU}. Thus
$$\phi_{P_V P_U}(\varepsilon)=\lambda_{\min}\left(\left(((1+\varepsilon) \mathrm{Id}-P_V P_U)^{-1}\right)_S\right)=\frac{4 \varepsilon-\sqrt{2}+3}{4 \varepsilon^2+6 \varepsilon+2}$$
and
$$\lim _{\varepsilon \rightarrow 0^{+}} \lambda_{\min }\left(\left(\left((1+\varepsilon) \mathrm{Id}-P_V P_U\right)^{-1}\right)_S\right)=\frac{3-\sqrt{2}}{2}.$$ By Theorem \ref{angle theorem2} we have $$
\begin{aligned}c_F(U, V) &=\frac{1}{2 \lim _{\varepsilon \rightarrow 0^{+}} \lambda_{\min }\left(\left(\left((1+\varepsilon) \mathrm{Id}-P_V P_U\right)^{-1}\right)_S\right)-1}-1
\\ & =\frac{\sqrt{2}}{2},
\end{aligned}$$
which implies that the Friedrichs angle between $U$ and $V$ is $\pi/4$. See Figure~\ref{f:angle} below.
\end{example}

\begin{figure}[H]
\centering
\begin{tikzpicture}[scale=2.5]

\draw[thick,->] (-1/2,0,0) -- (2,0,0) node[anchor=north east]{$x$};
\draw[ultra thick,->] (0,-1/2,0) -- (0,2.1,0) node[anchor=north west]{$y$};
\draw[thick,->] (0,0,-1/2) -- (0,0,2.4) node[anchor=south]{$z$};

\filldraw[blue!20,opacity=0.6] (-1/4,-1/2,0) -- (1.8,-1/2,0) -- (1.8,1.8,0) -- (-1/4,1.8,0) -- cycle;
\node[blue!40!black] at (1.3,1.5,0) {$U: z=0$};

\filldraw[red!30,opacity=0.6] (-1/4,-1/2,-1/2) -- (-1/8,2,2) -- (1.94,2,2) -- (1.76,-1/2,-1/2) -- cycle;
\node[red!70!black] at (1.3,1,1) {$V: z=y$};

\filldraw[blue!40,opacity=0.6]
  (0,-1/2,-0.33) -- (0,1.8,-0.33) -- (0,1.8,2.1) -- (0,-1/2,2.1) -- cycle;
\node[blue!80!black] at (-0.5,0.3,1.0) {$(U\cap V)^{\perp}: x=0$};


\draw[very thick,blue!100!black] (0,-0.7,0) -- (0,2,0);
\node[blue!85!black,anchor=east] at (0.3,2.2,0) {$U\cap (U\cap V)^{\perp}$};

\draw[very thick,red!85!black]
(0,-0.75,-3/5) -- (0,2.1,1.7);
\node[red!85!black,anchor=west] at (-0.4,2.5,2.5) {$V\cap (U\cap V)^{\perp}$};

\def\angrad{0.30} 
\draw[very thick,purple!80!black]
  plot[domain=0:40, samples=60, variable=\t]
    ({0}, {\angrad*cos(\t)}, {\angrad*sin(\t)});
\node[purple!80!black,inner sep=1pt]
  at (0, {1.35*\angrad*cos(22.5)}, {0.85*\angrad*sin(22.5)}) {$\theta$};
\end{tikzpicture}
 \caption{Friedrichs angle between two planes}
  \label{f:angle}
\end{figure}
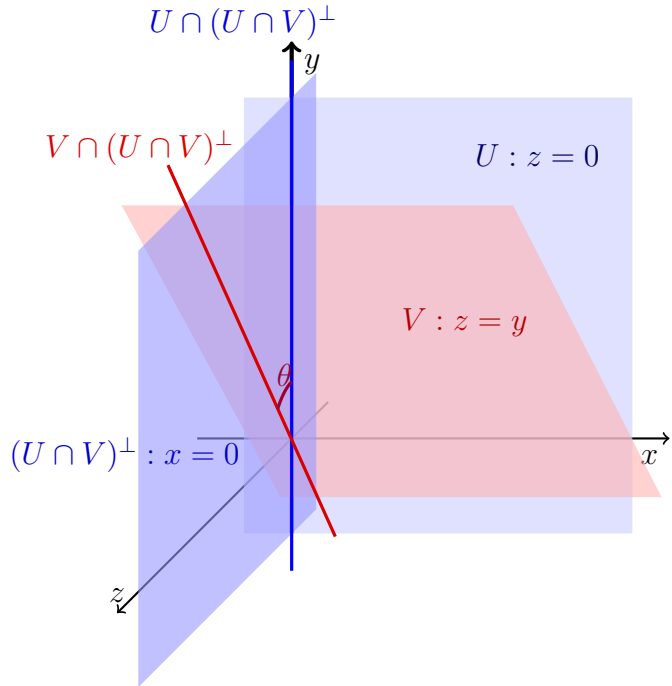

\begin{remark} See \cite[Lemma 9.5(7)]{DFbook} for finding the cosine of the Friedrichs angle
between two subspaces
via norms of various matrices of projections.
\end{remark}

\section{Nonlinear results}\label{s:nonlinear}

In this section, we establish further nonlinear results concerning the modulus of
conically averaged mappings, which may be viewed as generalizations or alternatives to the preceding results.
We start with a stability result on conically averaged mappings, which significantly improves
Lemma~\ref{l:minus}.

\subsection{A nonlinear inequality}

\begin{theorem}\label{T-gammaId}
Let $T: X \rightarrow X$ be conically averaged and let $\gamma \geq 0$. Then
$$
k(T-\gamma \Id) \leq k(T)+\frac{\gamma}{2};$$
consequently,
$T-\gamma\Id$ is conically averaged. In particular, $k(T-\mathrm{Id}) \leq k(T)+1/2$.
\end{theorem}

\begin{proof}
If $k(T)=0$, then $T=\operatorname{Id}+v$ for some $v \in X$ by Lemma~\ref{zero case}, hence $T-\gamma \Id=(1-\gamma)\Id+v$. Thus $k(T-\gamma \Id)=k((1-\gamma)\Id)=\gamma/2$ by Lemma~\ref{zero case}
 and Example \ref{e:Id}.

Now assume $k(T) \in(0,+\infty)$. Then $T$ is conically $k(T)$-averaged by Lemma~\ref{zero case}. By definition \begin{equation}\label{e1:T-gammaId}
T=(1-k(T)) \operatorname{Id}+k(T) N
\end{equation} for some nonexpansive operator $N$. Set \begin{equation}\label{e2:T-gammaId}
N^{\prime}:=\frac{(T-\gamma \Id)-\left[1-\left(k(T)+\frac{\gamma}{2}\right)\right] \Id}{k(T)+\frac{\gamma}{2}}.
\end{equation}
Plugging \eqref{e1:T-gammaId} into \eqref{e2:T-gammaId} gives $$N^{\prime}=\frac{k(T) N-\frac{\gamma}{2} \Id}{k(T)+\frac{\gamma}{2}}.$$ Using Triangle inequality and the nonexpansiveness of $N$, we have $$
(\forall x, y \in X)\ \left\|N^{\prime} x-N^{\prime} y\right\| \leq \frac{k(T)+|-\frac{\gamma}{2}|}{k(T)+\frac{\gamma}{2}}\|x-y\|=\|x-y\|,
$$which implies that $N^{\prime}$ is nonexpansive. Note that
$$T-\gamma \Id=\left[1-\left(k(T)+\frac{\gamma}{2}\right)\right] \Id+\left(k(T)+\frac{\gamma}{2}\right) N^{\prime}.$$ Thus, $T-\gamma \Id$ is conically $[k(T)+\frac{\gamma}{2}]$-averaged, which gives
$$
k(T-\gamma \mathrm{Id}) \leq k(T)+\frac{\gamma}{2}.
$$
\end{proof}

Using Theorem \ref{T-gammaId}, we can extend the key Lemma \ref{singular lemma} to the nonlinear case.

\begin{corollary}\label{c:negative:t}
Let $T: X \rightarrow X$ be conically averaged. Then $\lim _{\varepsilon \rightarrow 0^{+}} k(T-\varepsilon \mathrm{Id})$ exists, and
\begin{equation}\label{e:minus:epsilon}
k(T)=\lim _{\varepsilon \rightarrow 0^{+}} k(T-\varepsilon \mathrm{Id}).
\end{equation}
\end{corollary}

\begin{proof}
By Theorem~\ref{T-gammaId},
\begin{equation}\label{e:holiday1}
\varlimsup_{\varepsilon \rightarrow 0^{+}} k(T-\varepsilon \Id) \leq k(T).
\end{equation}

In view of Lemma~\ref{zero case} and Example \ref{e:Id}, we may assume $k(T-\varepsilon \Id)\neq0$ for every $\varepsilon>0$, otherwise $T=(1+\varepsilon)\Id+v$ is not conically averaged. From
\eqref{e:holiday1} we see that $T-\varepsilon \Id$ is conically averaged for sufficiently small $\varepsilon>0$. Hence, by Lemma~\ref{l:char:con}, for any $x, y \in X$,
\begin{equation}\label{e:T-eId}
\|(T-\varepsilon \Id)  x-(T-\varepsilon \Id) y-(x-y)\|^2\leq 2 k(T-\varepsilon \Id)\left(\|x-y\|^2-\langle x-y,(T-\varepsilon \Id) x-(T-\varepsilon \Id) y\rangle\right).
\end{equation}
Note that $\varliminf_{\varepsilon \rightarrow 0^{+}} k(T-\varepsilon \mathrm{Id})$ is finite due to its nonnegativity and \eqref{e:holiday1}. Taking $\varliminf$ as $\varepsilon\rightarrow 0^+$ in \eqref{e:T-eId},
we obtain
$$
\begin{aligned}
\|T x-T y-(x-y)\|^2 & \leq 2 \varliminf_{\varepsilon\rightarrow 0^{+}} [k(T-\varepsilon\Id)\left(\|x-y\|^2-\langle x-y,(T-\varepsilon\Id) x-(T-\varepsilon\Id) y\rangle\right)] \\& \leq 2 \varliminf_{\varepsilon\rightarrow 0^{+}}k(T-\varepsilon\Id)\left(\|x-y\|^2-\langle x-y,T x-T y\rangle\right)
\end{aligned}
$$
by the continuity of norm and inner product, and Fact~\ref{f:l:supinf}.
It follows that
\begin{equation}\label{e:holiday2}
k(T)\leq \varliminf_{\varepsilon\rightarrow 0^{+}}k(T-\varepsilon\Id)
\end{equation}
by Corollary \ref{c:char:con}. The result follows by combining \eqref{e:holiday1} and \eqref{e:holiday2}.
\end{proof}

Following Giselsson \cite[Definition 3.7]{PG},
we define negatively conically averaged mappings.
\begin{definition} We say that $T:X\rightarrow X$ is negatively conically $\alpha$-averaged if $-T$ is conically
$\alpha$-averaged.
\end{definition}

\begin{corollary}\label{c:estimate}
Let $T: X \rightarrow X$. Suppose that $\Id-T$ is conically averaged. Then $$
k(-T) \leq k(\Id-T)+\frac{1}{2},
$$
i.e., $T$ is negatively conically $[k(\Id-T)+1/2]$-averaged.
\end{corollary}

\begin{proof}
Apply Theorem \ref{T-gammaId} with $T$ replaced by $\Id-T$ and with $\gamma=1$.
\end{proof}

The above result is particularly useful in Section~\ref{s:hypoconvex}.

\subsection{A continuity result}\label{s:sub:continuity}

\begin{theorem}\label{1-conti}
Let $T: X \rightarrow X$ be conically averaged. Then $\lim _{\alpha \rightarrow 1^-} k(\alpha T)$ exists, and $$k(T)=\lim _{\alpha \rightarrow 1^-} k(\alpha T).$$
\end{theorem}

\begin{proof}
For every $\alpha \in[0,1]$, we have $k(\alpha T) \leq \alpha k(T)+(1-\alpha)/2$ by Corollary \ref{scalarmul}. Taking $\varlimsup$ when $\alpha \rightarrow 1^-$ yields
\begin{equation}\label{e:scalar1}
\varlimsup_{\alpha \rightarrow 1^-} k(\alpha T) \leq k(T).
\end{equation}

In view of Lemma~\ref{zero case} and Example \ref{e:Id}, we may assume $k(\alpha T)\neq0$ for every $\alpha\in(0,1)$, otherwise $T=\alpha^{-1}(\Id+v)$ is not conically averaged. From \eqref{e:scalar1} we see that $\alpha T$ is conically $k(\alpha T)$-averaged for all $\alpha$ less than and sufficiently nearby $1$.
Hence, by Lemma~\ref{l:char:con},  for any $x, y \in X$,
\begin{equation}\label{e:alT}
\|\alpha T x-\alpha T y-(x-y)\|^2 \leq 2 k(\alpha T)\left(\|x-y\|^2-\langle x-y,\alpha T x-\alpha T y\rangle\right).
\end{equation}
Note that $\varliminf_{\alpha \rightarrow 1^{-}} k(\alpha T)$ is finite due to its nonnegativity and \eqref{e:scalar1}. Taking $\varliminf$ when $\alpha \rightarrow 1^-$ in \eqref{e:alT}, we obtain
$$
\begin{aligned}
\|T x-T y-(x-y)\|^2 & \leq 2 \varliminf_{\alpha \rightarrow 1^{-}} [k(\alpha T)\left(\|x-y\|^2-\langle x-y,\alpha T x-\alpha T y\rangle\right)] \\& \leq 2 \varliminf_{\alpha \rightarrow 1^{-}} k(\alpha T)\left(\|x-y\|^2-\langle x-y,T x-T y\rangle\right)
\end{aligned}
$$
due to the continuity of norm and inner product and Fact~\ref{f:l:supinf}.
It follows that
\begin{equation}\label{e:scalar2}
k(T) \leq \varliminf_{\alpha \rightarrow 1^{-}} k(\alpha T)
\end{equation}
by Corollary \ref{c:char:con}. Combining \eqref{e:scalar1} and \eqref{e:scalar2} gives
$$
k(T) \leq \varliminf_{\alpha \rightarrow 1^{-}} k(\alpha T) \leq \varlimsup_{\alpha \rightarrow 1^{-}} k(\alpha T) \leq k(T),
$$
and all inequalities turn into equalities.
\end{proof}

We show next that this continuity result essentially provides an alternative characterization of conically averaged matrices, in comparison with Theorem~\ref{mainresult1}.

\begin{lemma}\label{l:invertiblem}
 Let $A \in M_n(\mathbb{R})$ and $\alpha_{0}> 0$. Then the following hold:
\begin{enumerate}
\item\label{i:less1} $(\exists\delta>0)(\forall \alpha\in (\alpha_{0}-\delta,\alpha_{0}))\ \Id -\alpha A \text{ is invertible}.$
\item\label{i:big1} $(\exists\delta>0)(\forall \alpha\in (\alpha_{0}, \alpha_{0}+\delta))\ \Id -\alpha A \text{ is invertible}.$
\end{enumerate}
\end{lemma}
\begin{proof}  By choosing $\delta>0$ small, we can assume $\alpha>0$ when
$\alpha\in (\alpha_{0}-\delta, \alpha_{0}+\delta)$.
Write $\Id -\alpha A=\alpha(1/\alpha\Id-A)$.
To show $\Id -\alpha A$ invertible, it suffices to show $1/\alpha\not\in\sigma(A)$.
Since the proof for \ref{i:big1} is similar, we prove \ref{i:less1} only.

We prove \ref{i:less1} by considering two cases. Observe that $\sigma (A)$ has at most $n$ elements, including the complex eigenvalues. Let us view each element in $\sigma(A)$ as a point in the complex plane.

Case 1: $1/\alpha_{0}\not\in\sigma(A)$. In the complex plane, the distance from the point $(1/\alpha_{0},0)$ to
the set $\sigma(A)$ is positive. Then we can choose $\delta>0$ sufficiently small such that
$1/\alpha$ is nearby $1/\alpha_{0}$ and
$1/\alpha\not\in \sigma(A)$.

Case 2: $1/\alpha_{0}\in\sigma(A)$. In the complex plane, the distance from the point $(1/\alpha_{0},0)$ to
the set $\sigma(A)\setminus\{1/\alpha_{0}\}$ is positive. Because $\alpha\in (\alpha_{0}-\delta,\alpha_{0})$, we
have $1/\alpha>1/\alpha_{0}$. Also we can
choose $\delta>0$ sufficiently small such that $1/\alpha$ is nearby $1/\alpha_{0}$ and
$1/\alpha\not\in (\sigma(A)\setminus \{1/\alpha_{0}\})$.
Then $1/\alpha\not\in\sigma(A)$.
\end{proof}

Fix $A \in M_n(\mathbb{R})$. By virtue of Lemma~\ref{l:invertiblem},
we can define
\begin{equation}\label{e:psai:c}
(\forall \alpha \in (c,1))\ \psi_A(\alpha):=\lambda_{\min }\left(\left((\operatorname{Id}-\alpha A)^{-1}\right)_S\right),
\end{equation}
where $c<1$ is chosen sufficiently nearby $1$ such that
$\operatorname{Id}-\alpha A$ is invertible for every $\alpha \in (c,1)$.

Armed with Theorem~\ref{nonsingular theorem}, Theorem~\ref{1-conti} and Lemma~\ref{l:invertiblem}, we can now give another characterization of conically averaged matrices.

\begin{theorem}\emph{\textbf{(conically averaged matrix: characterization III)}}\label{alterchar}
Let $A \in M_n(\mathbb{R})$. Then $A$ is conically averaged if and only if $\lim _{\alpha \rightarrow 1^{-}} \psi_A(\alpha)$ exists and belongs to $(0, +\infty]$, in which case the formula holds: $$
k(A)=\frac{1}{2} \frac{1}{\lim _{\alpha \rightarrow 1^{-}} \psi_A(\alpha)}.
$$
\end{theorem}

\begin{proof}
The proof parallels that of Theorem~\ref{mainresult1}. With the given $A \in M_n(\mathbb{R})$,
define
$c<1$ and $\psi_{A}$ as in equation~\eqref{e:psai:c}.

``$\Rightarrow$": Suppose $A$ is conically averaged.
For every $\alpha \in(c, 1)$, we have $\operatorname{Id}-\alpha A$ is invertible. Moreover, $\alpha A$ is conically averaged by Corollary \ref{scalarmul}. Then Theorem \ref{nonsingular theorem} yields $$
k(\alpha A)=\frac{1}{2} \frac{1}{\lambda_{\min }\left(\left(( \mathrm{Id}-\alpha A)^{-1}\right)_S\right)}=\frac{1}{2} \frac{1}{\psi_A(\alpha)} .
$$
Since $A$ is conically averaged, by Theorem \ref{1-conti} we have $\lim _{\alpha \rightarrow 1^{-}} k(\alpha A)$ exists and $$
k(A)=\lim _{\alpha \rightarrow 1^{-}} k(\alpha A)=\frac{1}{2} \frac{1}{\lim _{\alpha \rightarrow 1^{-}} \psi_A(\alpha)}\in[0, +\infty).
$$
Thus, $\lim _{\alpha \rightarrow 1^{-}} \psi_A(\alpha)$ exists and belongs to $(0, +\infty]$, in which case the formula holds.

``$\Leftarrow$":
Suppose $\lim _{\alpha \rightarrow 1^{-}} \psi_A(\alpha)$ exists and belongs to $(0, +\infty]$.
Then by the property of one-sided limit, there exists $c<d<1$ such that for every $\alpha \in (d, 1)$: $\psi_A(\alpha)>0$, i.e., $\mathrm{Id}-\alpha A$ is invertible and $((\mathrm{Id}-\alpha A)^{-1})_S$ is positive definite. Applying Theorem \ref{nonsingular theorem} we have that $\alpha A$ is conically averaged for any $\alpha \in (d, 1)$ and $$
k(\alpha A)=\frac{1}{2} \frac{1}{\lambda_{\min }\left(\left((\operatorname{Id}-\alpha A)^{-1}\right)_S\right)}=\frac{1}{2} \frac{1}{\psi_A(\alpha)}.
$$
Since $\lim _{\alpha \rightarrow 1^{-}} \psi_A(\alpha) \in (0, +\infty]$, taking limit we have
\begin{equation}\label{e:finite:mod1}
\lim _{\alpha \rightarrow 1^{-}} k(\alpha A)=\frac{1}{2} \frac{1}{\lim _{\alpha \rightarrow 1^{-}} \psi_A(\alpha)}\in[0,+\infty).
\end{equation}
Applying Proposition \ref{usefullinear}\ref{i:linear3} to the conically averaged operator $\alpha A$, we have
$$
(\forall z\in \mathbb{R}^n)\ \|(\alpha A) z\|^2+\left(1-2 k(\alpha A)\right)
\|z\|^2 \leq 2(1-k(\alpha A))\langle z,(\alpha A) z\rangle.
$$
Sending $\alpha \rightarrow 1^{-}$ gives
\begin{equation}\label{e:finite:mod2}
(\forall z \in \mathbb{R}^n)\ \|A z\|^2+(1-2\lim_{\alpha \rightarrow 1^{-}} k(\alpha A))\|z\|^2
\leq 2(1-\lim _{\alpha \rightarrow 1^{-}} k(\alpha A))\langle z,A z\rangle
\end{equation}
due to the continuity of norm and inner product.
In view of \eqref{e:finite:mod1} and \eqref{e:finite:mod2}, we deduce that
$A$ is conically $\lim_{\alpha \rightarrow 1^{-}}k(\alpha A)$-averaged by applying Proposition \ref{usefullinear}\ref{i:linear3} again. Altogether, we complete the proof.
\end{proof}

\begin{corollary}\emph{\textbf{(matrix with nonzero fixed point: formula II)}}
Let $A \in M_n(\mathbb{R})$ and suppose that $\mathrm{Id}-A$ is not invertible (i.e., $1 \in \sigma(A)$). Then $A$ is conically averaged if and only if $\lim _{\alpha \rightarrow 1^{-}} \lambda_{\min }\left(\left(( \mathrm{Id}-\alpha A)^{-1}\right)_S\right)$ exists and belongs to $(0,+\infty]$, in which case the formula holds:
$$
k(A)=\frac{1}{2} \frac{1}{\lim _{\alpha \rightarrow 1^{-}} \lambda_{\min }\left(\left((\operatorname{Id}-\alpha A)^{-1}\right)_S\right)}.
$$
\end{corollary}

\begin{corollary}\emph{\textbf{(Friedrichs angle: formula II)}}\label{alterFriedrichs}
Let $U, V$ be linear subspaces of $\RR ^n$. If $U \cap V \neq \{0\}$ and $U \cap V \neq \mathbb{R}^n$, then$$
c_F(U, V)=\frac{1}{2 \lim _{\alpha \rightarrow 1^{-}} \lambda_{\min }\left(\left(( \mathrm{Id}-\alpha P_V P_U)^{-1}\right)_S\right)-1}-1.
$$
\end{corollary}

\section{Hypoconvex functions}\label{s:hypoconvex}
Conically averaged mappings are very useful for characterizing the proximal and
reflection mappings of hypoconvex functions. Recall that
for a function $f:X\rightarrow\RX$ and $\mu\in\RPP$, the proximal mapping of $f$ is defined by
$$(\forall x\in X)\ \prox{f}(x):=\argmin_{y\in X}\left\{f(y)+\frac{1}{2\mu}\|y-x\|^2\right\},$$
and the reflection mapping of $f$ is defined by $\refl{f}:=2\prox{f}-\Id$.

\begin{definition} For a function $f:X\rightarrow\RX$ and $\sigma\geq 0$, we say that $f$ is $\sigma$-hypoconvex
if $$f+\sigma\frac{\|\cdot\|^2}{2} \text{ is convex.} $$
A convex function is just the $0$-hypoconvex function.
\end{definition}

For a hypoconvex function $f$, possibly nonconvex, its Clarke subdifferential \cite{Clarke}
and Morduknovich limiting subdifferential \cite{Boris} coincide. We just write $\partial_{\sharp} f$ for both of them.
Below we use the convention $1/0 :=+\infty$.

\begin{fact}\emph{\cite[Propositions 6.3, 6.4]{BMW}}\label{f:hypoconvex}
Let $f:X\rightarrow\RX$ be $\sigma$-hypoconvex with $\sigma\geq 0$.
Then the following hold:
\begin{enumerate}
\item The subdifferential of possibly nonconvex $f$ is
$$\partial_{\sharp} f=\partial\bigg(f+\frac{\sigma}{2}\|\cdot\|^2\bigg)-\sigma\Id.$$
\item $\partial_{\sharp} f$ is maximally $(-\sigma)$-monotone.
\item If $\mu\in (0,1/\sigma)$, then $$\prox{f}=J_{\mu\partial_{\sharp} f}:=(\Id+\mu\partial_{\sharp} f)^{-1}=
(\Id+\partial_{\sharp}(\mu f))^{-1}=J_{\partial_{\sharp}(\mu f)}$$
is $1/(1-\mu\sigma)$-Lipschitz on $X$.
\end{enumerate}
\end{fact}

We also need the following result, which improves Proposition~\ref{newOYprop} when $T_{1}=\Id$.
\begin{lemma} \label{l:reflect}
Let $T:X\rightarrow X$ and $\lambda\in (0,+\infty)$. Then $T$ is conically $\alpha$-averaged
if and only if $(1-\lambda)\Id+\lambda T$ is conically $(\lambda\alpha)$-averaged.
Consequently,
$k((1-\lambda)\Id+\lambda T)=\lambda k(T)$.
\end{lemma}
\begin{proof}
$T$ is conically $\alpha$-averaged if $T=(1-\alpha)\Id+\alpha N$ where $N:X\rightarrow X$ is nonexpansive.
The result follows from
\begin{align}
T=(1-\alpha)\Id+\alpha N &\Leftrightarrow (1-\lambda)\Id+\lambda T = (1-\lambda)\Id+\lambda[((1-\alpha)\Id+\alpha N)]\\
& \Leftrightarrow (1-\lambda)\Id+\lambda T=(1-\lambda\alpha)\Id+\lambda\alpha N;
\end{align}
see, e.g., \cite[Proposition 2.2(ii)]{BDP}.
\end{proof}

Our final main result concerns the proximal and reflection mappings of hypoconvex functions.
Although Theorem~\ref{t:hypo:convex}\ref{i:inverse1} is known, to the best of our knowledge,
Theorem~\ref{t:hypo:convex}\ref{i:inverse2}--\ref{i:primal1} are new.

\begin{theorem}\label{t:hypo:convex} Let $f:X\rightarrow\RX$ be $\sigma$-hypoconvex with $\sigma\geq 0$, and let
$\mu\in (0,1/\sigma)$. Then the following
hold:
\begin{enumerate}
\item\label{i:inverse1} $J_{\partial_{\sharp}(\mu f)^{-1}}$ is conically $(1/[2(1-\mu\sigma)])$-averaged. Consequently, $k(J_{\partial_{\sharp}(\mu f)^{-1}})\leq 1/[2(1-\mu\sigma)]$.
\item\label{i:inverse2} $R_{\partial_{\sharp}(\mu f)^{-1}}$ is conically $(1/(1-\mu\sigma))$-averaged. Consequently, $k(R_{\partial_{\sharp}(\mu f)^{-1}})\leq 1/(1-\mu\sigma)$.
\item\label{i:primal2} $\refl{f}$ is negatively conically $(1/(1-\mu\sigma))$-averaged. Consequently, $k(-\refl{f})\leq 1/(1-\mu\sigma)$.
\item\label{i:primal1} $\prox{f}$ is negatively conically $(1/[2(1-\mu\sigma)]+1/2)$-averaged. Consequently, $k(-\prox{f})\leq 1/[2(1-\mu\sigma)]+1/2$.
\end{enumerate}
\end{theorem}
\begin{proof} Recall the resolvent identity \cite[Lemma 12.14]{RW}:
Every set-valued mapping $A:X\rightrightarrows X$ obeys
\begin{equation}\label{e:res:identity}
J_{A^{-1}}=\Id-J_{A}.
\end{equation}
(Although \cite[Lemma 12.14]{RW} is stated in $\RR^n$ there, it actually holds in a Hilbert space with the same proof.)
Equation \eqref{e:res:identity}, together with Fact~\ref{f:hypoconvex}, gives
\begin{equation}
J_{\partial_{\sharp}(\mu f)^{-1}}=\Id -(\Id+\partial_{\sharp}(\mu f))^{-1}=\Id-\prox{f}, \text{ and }
\end{equation}
\begin{equation}\label{e:primal:dual}
R_{\partial_{\sharp}(\mu f)^{-1}}=2J_{\partial_{\sharp}(\mu f)^{-1}}-\Id=\Id-2\prox{f}=-\refl{f}.
\end{equation}

\ref{i:inverse1} Apply \cite[Theorem 6.5(ii)]{BMW}. Or use Proposition~\ref{JACACOMONOPROP} because
$\partial_{\sharp}(\mu f)^{-1}$ is $(-\mu\sigma)$-comonotone and $-\mu\sigma>-1$.

\ref{i:inverse2} Combine \ref{i:inverse1}, \eqref{e:primal:dual}, and Lemma~\ref{l:reflect} with $\lambda=2$.

\ref{i:primal2} Use \ref{i:inverse2} and \eqref{e:primal:dual}.

\ref{i:primal1} Since $J_{\partial_{\sharp}(\mu f)^{-1}}=\Id -\prox{f}$ is conically averaged by \ref{i:inverse1},
Corollary~\ref{c:estimate} implies
\begin{equation}
k(-\prox{f})\leq \frac{1}{2(1-\mu\sigma)}+\frac{1}{2},
\end{equation}
so $-\prox{f}$ is conically averaged. Or write
$$-\prox{f}=\frac{-\refl{f}-\Id}{2}.$$
Then apply \ref{i:primal2},
Proposition~\ref{newOYprop}, and Example~\ref{e:Id}.
\end{proof}

\begin{remark} Astute readers may find that when $\sigma=0$, Theorem~\ref{t:hypo:convex}\ref{i:primal1}
only gives
that $\prox{f}$ is nonexpansive by Fact~\ref{f:one:lip}, instead of firmly nonexpansive.
We emphasize that the proximal mapping of a function is nonexpansive if and only if it
is firmly nonexpansive; see,
e.g., \cite[Theorem 3.2]{luo}.
\end{remark}

We finish the paper with an example illustrating that the estimations given in Theorem~\ref{t:hypo:convex} are tight!

\begin{example}[quadratic functions] Let
$Q\in \mathbb{S}^n$, $b\in\mathbb R^n$, and $c\in\mathbb R.$
Define the quadratic function
$$
f:\RR^n\rightarrow\RR: x\mapsto \frac12\langle Qx,x\rangle+\langle b,x\rangle+c.
$$
Set
$$
q:=\lambda_{\min}(Q).
$$
Then $\partial_\sharp f(x)=\nabla f(x)=Qx+b$ and by Proposition \ref{linear mc}, $m(\partial_\sharp f)=q$. In particular, $f$ is hypoconvex with the sharp hypoconvexity parameter
$$
        \sigma=\max\{0,-q\}.
$$

Let $\mu>0$ with $1+\mu q>0$. Then $\Id+\mu Q$ is positive definite, and hence
\begin{equation}\label{e:hypo:prox}
        \mathrm{P}_\mu f(x)=J_{\partial_{\sharp}(\mu f)}=(\Id+\mu Q)^{-1}(x-\mu b).
\end{equation}
We now compute the $k$-modulus of $J_{\partial_{\sharp}(\mu f)^{-1}}=\mathrm{Id}-\mathrm{P}_\mu f$, $R_{\partial_{\sharp}(\mu f)^{-1}}=-\mathrm{R}_\mu f$ and $-\mathrm{P}_\mu f$.
It follows from \eqref{e:hypo:prox} that
$$J_{\partial_{\sharp}(\mu f)^{-1}}=\left(\Id-(\Id+\mu Q)^{-1}\right)(x-\mu b)+\mu b,$$
$$R_{\partial_{\sharp}(\mu f)^{-1}}=\left(\Id-2(\Id+\mu Q)^{-1}\right)(x-\mu b)+\mu b,$$
and
$$-P_{\mu}f(x)=-(\Id+\mu Q)^{-1}(x-\mu b).$$
Since the modulus is translation invariant by Lemma \ref{zero case}, it suffices to consider only the linear part. Let
$$E:=\Id-(\mathrm{Id}+\mu Q)^{-1}.$$
Then $E$ is also symmetric and its eigenvalues are
$$
1-\frac{1}{1+\mu \lambda}=\frac{\mu \lambda}{1+\mu \lambda}, \quad \lambda \in \sigma(Q).
$$
The scalar function
$$
\lambda \mapsto \frac{\mu \lambda}{1+\mu \lambda}
$$
is increasing on $\menge{\lambda}{1+\mu \lambda>0}$. Hence
$$
\lambda_{\min }(E)=\frac{\mu q}{1+\mu q} .
$$
Apply Proposition \ref{kSn} to obtain
\begin{equation}\label{e:minh1}
k\big(J_{\partial_\sharp(\mu f)^{-1}}\big)=k(E)=\frac{1-\lambda_{\min }(E)}{2}=\frac{1}{2(1+\mu q)}.
\end{equation}
Similarly, we have
\begin{equation}\label{e:minh2}
k\big(R_{\partial_\sharp(\mu f)^{-1}}\big)=k(-\mathrm{R}_\mu f)=k(\Id-2(\Id+\mu Q)^{-1})=\frac{1}{1+\mu q},
\end{equation}
and
\begin{equation}\label{e:minh3}
k(-\mathrm{P}_\mu f)=k(-(\Id+\mu Q)^{-1})=\frac12+\frac{1}{2(1+\mu q)}.
\end{equation}

Thus, in the genuinely nonconvex case $q<0$, taking the sharp parameter $\sigma=-q$ yields exactly the estimation given by Theorem \ref{t:hypo:convex}, so the bounds in Theorem \ref{t:hypo:convex} are tight.
Moreover, $J_{\partial_{\sharp}(\mu f)}$ is not conically averaged because $(\Id+\mu Q)^{-1}$ has an eigenvalue
$1/(1+\mu q)>1$. In the convex case $q \geq 0$, Theorem \ref{t:hypo:convex} with $\sigma=0$ gives valid but generally nonsharp bounds. In fact, for convex quadratic functions, we have the modulus of conical averagedness given by
\eqref{e:minh1}--\eqref{e:minh3}, which are sharper.
\end{example}

\section*{Acknowledgments}
The authors thank the editor and the referees for careful reading and constructive comments. This work originated in part from the second author's Master thesis at the University of British Columbia.
H.\ Luo was partially supported by the NSF Grants of China and Chongqing (11991024, 12271071, KJZD-K 202500507). S.\ Song and X.\ Wang were partially supported by the Natural Sciences and Engineering Research Council of Canada. S.\ Song also acknowledges the supports from the Research Assistantship of Chongqing Normal University and the Melbourne Research Scholarship of the University of Melbourne.

\end{document}